\documentclass[11pt]{amsart}

\usepackage[colorlinks = true,
            linkcolor = red,
            urlcolor  = magenta,
            citecolor = black,
            anchorcolor = blue]{hyperref}
\usepackage{color}
\usepackage{comment}
\usepackage[shortalphabetic,initials,nobysame]{amsrefs}
\usepackage{upgreek}
\usepackage{tikz}
\usepackage[applemac]{inputenc} 
\usetikzlibrary{arrows}
\usepackage{xcolor}
\usepackage[all]{xy}
\usepackage{graphicx}
\usepackage{wrapfig}
\usepackage{yfonts}
\usepackage{enumerate}
\usepackage{graphicx}
\usepackage{amssymb}
\usepackage{epstopdf}
\usepackage{mathrsfs}
\usepackage[mathcal]{eucal}
\usepackage{bm}

\renewcommand{\eprint}[1]{\href{https://arxiv.org/abs/#1}{#1}}

\DeclareMathOperator{\diag}{diag}

\BibSpec{article}{%
    +{}  {\PrintAuthors}                {author}
    +{,} { \textit}                     {title}
     +{,} {}                            {note}
    +{.} { }                            {part}
    +{:} { \textit}                     {subtitle}
    +{,} { \PrintContributions}         {contribution}
    +{.} { \PrintPartials}              {partial}
    +{,} { }                            {journal}
    +{}  { \textbf}                     {volume}
    +{}  { \PrintDate}                {date}
    +{,} { \issuetext}                  {number}
    +{,} { \eprintpages}                {pages}
    +{,} { }                            {status}
    +{,} { \DOI}                   {doi}
    +{,} { \eprint}        {eprint}
      +{,} {\publisher}                {publisher}
    +{,} { \address}                   {address}
    +{}  { \parenthesize}               {language}
    +{}  { \PrintTranslation}           {translation}
    +{;} { \PrintReprint}               {reprint}
    +{.} {}                             {transition}
    +{}  {\SentenceSpace \PrintReviews} {review}
}

\newtheorem{Thm}{Theorem}[section]
\newtheorem{Lem}[Thm]{Lemma}
\newtheorem{Prop}[Thm]{Proposition}
\newtheorem{Cor}[Thm]{Corollary}

\theoremstyle{definition}
\newtheorem{Def}[Thm]{Definition}
\theoremstyle{remark}
\newtheorem{Rem}[Thm]{Remark}

\newtheoremstyle{named}{}{}{\itshape}{}{\bfseries}{.}{.5em}{#1 #3}
\theoremstyle{named}

\def\Q{\mathbb{Q}}

\def\g{\mathfrak{g}}
\def\Frenkel:2013uda{\mathfrak{h}}

\def\cC{\mathcal{C}}

\def\cL{\mathcal{L}}

\def\cV{\mathcal{V}}

\def\L{\Lambda}

\def\bo{\textbf{o}}

\def\=>{\Longrightarrow}

\def\to{\longrightarrow}

\def\o+{\oplus}
\def\bo+{\bigoplus}

\def\<{\langle}
\def\>{\rangle}
\def\({\left(}
\def\){\right)}

\def\^{\wedge}
\def\+{\dagger}

\def\dd[#1,#2]{\frac{d#1}{d#2}}
\def\del[#1,#2]{\frac{\partial #1}{\partial #2}}
\def\over[#1]{\overline{#1}}
\def\vec[#1]{\overrightarrow{#1}}

\makeatletter
\def\mr@ignsp#1 {\ifx\:#1\@empty\else #1\expandafter\mr@ignsp\fi}%
\newcommand{\multiref}[1]{\begingroup
\xdef\mr@no@sparg{\expandafter\mr@ignsp#1 \: }%
\def\mr@comma{}%
\@for\mr@refs:=\mr@no@sparg\do{\mr@comma\def\mr@comma{,}\ref{\mr@refs}}%
\endgroup}
\makeatother

\newcommand{\hypref}[2]{\ifx\href\asklFrenkel:2013udaas #2\else\href{#1}{#2}\fi}

\newcommand{\secref}[1]{Sec.~\multiref{#1}}

\newcommand{\figref}[1]{Fig.~\multiref{#1}}

\tikzset{->-/.style={decoration={
  markings,
  mark=at position .5 with {\arrow{latex}}},postaction={decorate}}}
\tikzset{
    >=latex
    }

\newcommand{\nc}{\newcommand}
\nc{\on}{\operatorname}
\nc{\la}{\lambda}
\nc{\wh}{\widehat}
\nc{\ghat}{\wh\g}
\nc{\mb}{\mathbf}

\begin{document}
\title{$q$-Opers and Bethe Ansatz for Open Spin Chains I}

\author[P. Koroteev]{Peter Koroteev}
\address{
\newline
Beijing Institute for Mathematical Sciences and Applications
\newline
Beijing, China
\newline
\href{mailto:peter.koroteev@gmail.com}{peter.koroteev@gmail.com}
}

\author[M. Shim]{Myungbo Shim}
\address{
\newline
          Yau Mathematical Sciences Center
          Tsinghua University
          \newline
          Beijing, China
        \newline
        \href{mailto: mbshim@tsinghua.edu.cn}{mbshim@tsinghua.edu.cn}
          }

\author[R. Singh]{Rahul Singh}
\address{
\newline
          Yau Mathematical Sciences Center
          Tsinghua University
        \newline
          Beijing, China 
          \newline
          \href{mailto: 95rahul32@gmail.com}{95rahul32@gmail.com}
          }

\date{\today}

\numberwithin{equation}{section}

\begin{abstract}
In a nutshell, the classical geometric $q$-Langlands duality can be viewed as a correspondence between the space of $(G,q)$-opers and the space of solutions of $^L\mathfrak{g}$ XXZ Bethe Ansatz equations. The latter describes spectra of closed spin chains with twisted periodic boundary conditions and, upon the duality, the twist elements are identified with the $q$-oper connections on a projective line in a certain gauge. In this work, we initiate the geometric study of Bethe Ansatz equations for spin chains with open boundary conditions. We introduce the space of $q$-opers whose defining sections are invariant under reflection through the unit circle in a selected gauge. The space of such reflection-invariant $q$-opers in the presence of certain nondegeneracy conditions is thereby described by the corresponding Bethe Ansatz problem. We compare our findings with the existing results in integrable systems and representation theory. This paper discusses the type-A construction leaving the general case for the upcoming work.
\end{abstract}

\maketitle

\setcounter{tocdepth}{1}
\tableofcontents

\section{Introduction}
The Geometric Langlands program has a multitude of layers with different levels of abstraction. In recent years, there has been significant progress in understanding the geometric $q$-Langlands correspondence \cites{Aganagic:2017la,KSZ,Frenkel:2020} which can be viewed as a $q$-deformation of the well-established duality that has recently been proved in a sufficiently general setting \cite{gaitsgory2025proofgeometriclanglandsconjecture}.
The $q$-version of the correspondence, however, is not yet at the same level of rigor.

Nevertheless, the ideas around $q$-Langlands have already led to a plethora of insightful results in mathematical physics, in particular, in understanding geometric structures of integrable systems as well as dualities between them. We refer the reader to recent review articles and lecture notes \cites{koroteev2023quantumgeometryintegrabilityopers,doi:10.1142/S0217751X24460102,zeitlin2024geometricrealizationsbetheAnsatz} and references therein.

\subsection{\texorpdfstring{$q$}{q}-Opers, \texorpdfstring{$QQ$}{QQ}-Systems and Bethe Ansatz}
One such statement relates the space of $(G,q)$-opers for a simple simply-connected complex Lie group $G$ and the space of solutions of $^L\mathfrak{g}$ XXZ Bethe Ansatz equations for a closed spin chain with twisted periodic boundary conditions.
In full generality it was studied in \cite{Frenkel:2020} where an explicit $q$-difference oper connection was constructed as a meromorphic section $A\in \text{Hom}(\mathcal{F}_G,\mathcal{F}^q_G)$, where $\mathcal{F}_G$ is a principal $G$-bundle over $\mathbb{P}^1$ and $\mathcal{F}^q_G$ is its pullback under the multiplicative $q$-shift of the base coordinate in the local trivialization $z\mapsto qz$. The $q$-oper condition imposes $A$ to take values in a Coxeter Bruhat cell for the Borel subgroup $B_-\subset G$ and the Miura condition ensures that $A$ preserves the reduction to $\mathcal{F}_{B_+}$. The $q$-oper is then equipped with the set of regular singularities which are described by polynomials $\Lambda_i(z)$ and twisted condition which requires $A$ to be $q$-gauge equivalent to a regular semisimple element of the torus $A(z)=g(qz)Zg(z)^{-1}$, where $Z\in H\subset H(z)$.

In the presence of certain mild nondegeneracy conditions it was proved in \cite{Frenkel:2020} that the space of $(G,q)$-opers admits two equivalent descriptions. The former is the system of finite $q$-difference equations on the family of polynomials $\{Q_\pm^i(z)\}_{i=1,\dots,r}$ referred to as the $QQ$-system for the Langlands dual group $^LG$. The Miura $(G,q)$-oper connection can thus be written as an element of $N_-(z)\prod_i \Lambda_i(z)^{\check{\alpha}_i}s_i N_-(z)\cap B_+(z)$. where $s_i$ are lifts of the Weyl group of  elements as
\begin{equation}
A(z)=\prod_i g_i(z)^{\check{\alpha}_i} \exp \frac{\Lambda_i(z)}{g_i(z)}\,,\qquad g_i(z)=\zeta_i\frac{Q_+^i(qz)}{Q_+^i(z)}\,.
\end{equation} 
The $QQ$-system equations can be viewed as obstructions to diagonalizing the above $q$-connection. These equations have been extensively studied in the literature on representation theory and quantum integrability.

The latter description is in terms of celebrated Bethe Ansatz equations for the XXZ spin chain which were studied in the representation theory literature \cites{Frenkel:2013uda,Frenkel:2016}. The Bethe equations describe the spectrum of $U_q(^L\hat{\mathfrak{g}})$ XXZ Hamiltonians (modulo some subtleties for non-simply laced groups, see \cite{Frenkel:2021vv}). The aforementioned Bethe Ansatz equations describe the spectrum of a \textit{closed} quantum spin Hamiltonian which acts on the tensor product of evaluation modules of $U_q(^L\hat{\mathfrak{g}})$ with \textit{twisted periodic} boundary conditions. 

On the left of Figure \ref{fig:foldedchain} tensor modules $V_1(a_1)\otimes\dots\otimes V_n(a_n)$ from the Hilbert space of the XXZ Hamiltonians are represented with vertical arrows on a cylinder. One of the representations $V_i(a_i)$ is intertwined with all the others via $R$-matrices of $U_q(^L\hat{\mathfrak{g}})$ (red circles) and in addition acted upon by the twist element $Z$. This action can be formalized in terms of the quantum Knizhnik-Zamolodchikov equation (qKZ) \cite{FR1998} $\Psi(q^\sigma a)=S_\sigma^{\text{closed}}\cdot \Psi(a)$ where $\Psi(a)$ is the state from the tensor product of the evaluation representations and $S_\sigma^{\text{closed}}$ stands for the composition of the $R$-matrices and the twist element $Z$. Here $\sigma$ labels the evaluation parameter which is being $q$-shifted.  In the $q\to 1$ limit the qKZ equations become the eigenvalue equations for operators $S_\sigma^{\text{closed}}$ and their spectrum can be obtained by solving Bethe Ansatz equations. 

\begin{figure}
\includegraphics[scale=0.26]{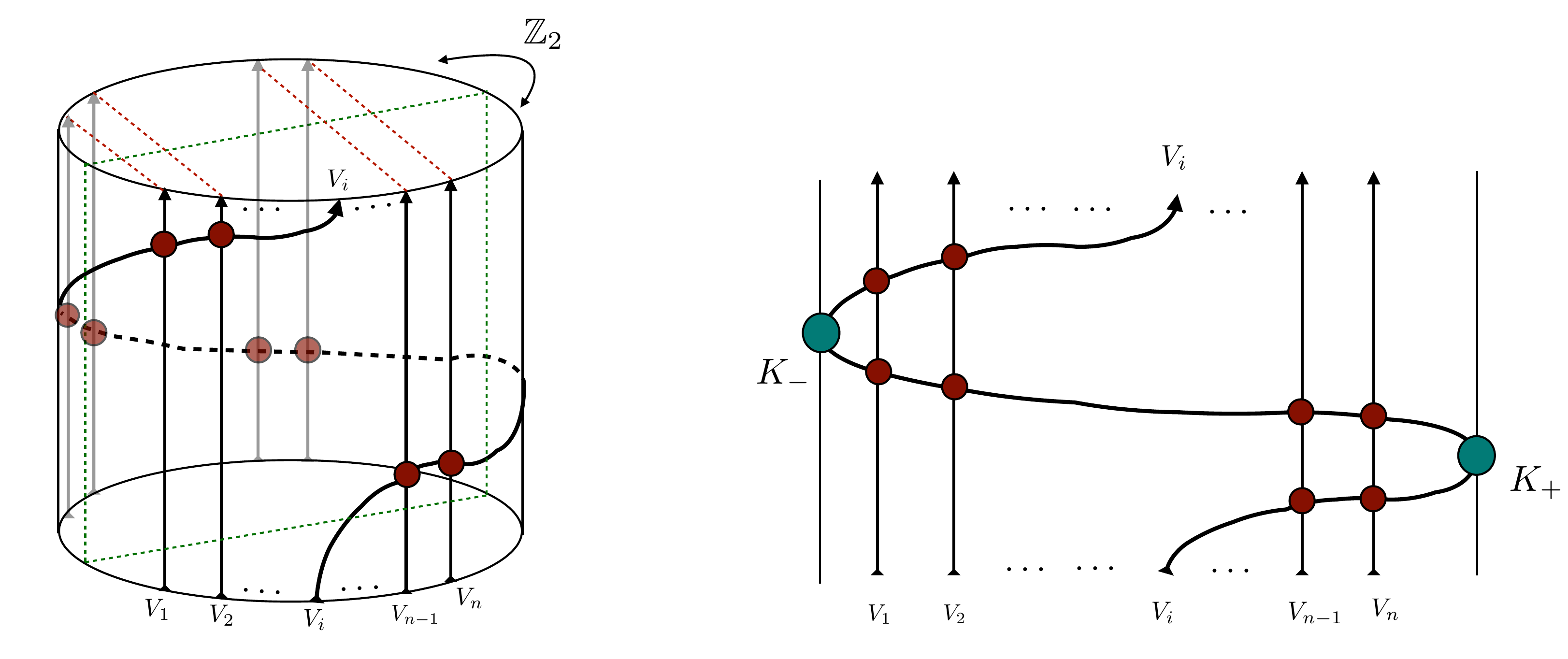}
\caption{Orbifolding trick yields an open spin chain from a closed one.}
\label{fig:foldedchain}
\end{figure}

In \cite{Okounkov:2015aa} and works that followed the qKZ connection and its solutions (as well as its mirror cousin dynamical equation where the twist variables are $q$-shifted) were understood geometrically through the equivariant quasimap counts to Nakajima quiver varieties. 
This approach in part inspires our work as well. In fact, it would be an interesting question for enumerative geometry to find the setup which realizes our Bethe Ansatz equations for open chains.

\subsection{Bethe Ansatz for Open Chains}
The right side of Figure \ref{fig:foldedchain} depicts the \textit{open} version of the qKZ equation known in the literature as the \textit{boundary qKZ equation} \cites{CheredInv,CheredNewOpen,Jimbo_1995,Reshetikhin:2013aa, Reshetikhin:2014aa}. Schematically, this equation can be written as 
$\Psi(q^\sigma a)=S_\sigma^{\text{open}}\cdot \Psi(a)$ where the new operator $S_\sigma^{\text{open}}$ contains the tensor product of $R$-matrices (red circles) as previously, but in addition, it also contains boundary $K$-matrices $K_+$ and $K_-$ for the right and left boundaries respectively. 
Together $K$ and $R$ matrices satisfy the so-called \textit{reflection equation} (boundary Yang Baxter equation).
The $K$-matrices depend on their own independent boundary parameters which will appear in the resulting Bethe equations for the open chain in question.

\subsection{Folding}
The left and right hand sides of Figure \ref{fig:foldedchain} can be related to each other via `folding' -- a certain $\mathbb{Z}/2\mathbb{Z}$ action on the data of the XXZ spin chain/qKZ equation. The figure shows that if the number of $V_i(a_i)$ terms in the state of the system is even, one can split them into two halves and identify the factors accordingly. In this way, we are selecting a certain $\mathbb{Z}/2\mathbb{Z}$-invarint subsector in the space of solutions of the qKZ equation. The folding trick also produces $K$ matrices at the orbifold points.

While the idea of folding has been discussed in the literature on integrable systems for a long time, to the best of our knowledge
it has never been implemented. The goal of this work is to provide a geometric construction of the folding trick and show that the Bethe Ansatz equations for the open XXZ chains in type $A$ can be derived from the space of so-called \textit{reflection-invariant $q$-opers}. We will continue this program for $(G,q)$-opers in the upcoming work.

In physics literature, in particular, in the context of the gauge/Bethe correspondence \cite{Nekrasov:2009uh}, open spin chains have been studied in \cites{Wang:2024aa,Kimura:2020aa,Frassek:2015aa}. We shall refer to some of those results in the paper.

\subsection{Structure of the Paper}
In the following section, we introduce the notion of reflection-invariant $(GL(2),q)$-opers and show that they are described by Bethe Ansatz equations for open chains. Section \ref{ref inv gln opers} provides a generalization of that construction to reflection-invariant $(GL(N),q)$-opers. In Section \ref{epsilon opers}, we address the folding construction for open XXX spin chains. The final Section \ref{Sec:Connect} describes in detail the match between Bethe equations obtained from our geometric construction and the equations from the existing literature.

\subsection{Acknowledgments}
We would like to thank Bart Vlaar for continuous discussions about the ideas around mathematical description of open spin chains over the last year. PK thanks Beijing Junior Talent Award Committee for financial support. PK also thanks T. Knight, C. Pacific, and Y. Zhang for support during the recent months. MS is supported by the Beijing Natural Science Foundation (BJNSF) grant IS24010 and by the Shuimu Scholar program of Tsinghua University.

\section{Reflection-invariant \texorpdfstring{$(GL(2),q)$}{(GL(2),q)}-Opers}\label{GL2}
We begin with developing our construction for $(GL(2),q)$-connections followed by the higher rank case in the next section. The beginning of this section goes along the lines of Section 3 of \cite{KSZ}. The construction follows the logic of $Z$-twisted Miura $q$-opers which were described using Bethe Ansatz for the XXZ spin chain.

Let $q\in\mathbb{C}^\times$ and not a root of unity. Consider the holomorphic vector bundle $E$ of rank $2$ over $\mathbb{P}^1$ that is equivariant with respect to the scaling of the base $M_q:\mathbb{P}^{1}\to \mathbb{P}^{1}$ that sends $z\mapsto qz$. Let $E^q$ denote the pullback of $E$ under $M_q$.

\begin{Def}
    For a Zariski open dense subset $U\subset\mathbb{P}^1$ consider $V=U\cap M_q^{-1}(U)$. A meromorphic $(GL(2),q)$-oper on $\mathbb{P}^1$ is a triple $(E,A,\mathcal{L})$, where $E$ is a rank-two holomorphic vector bundle on $\mathbb{P}^1$, $\mathcal{L}\subset E$ is a line subbundle, and $A\in\text{Hom}_{\mathcal{O}_{V}}(E,E^q)$ such that the restriction map
\begin{equation}
\label{eq:orbopdef}
\bar{A}: \mathcal{L}\to \left(E/\mathcal{L}\right)^{q}
\end{equation}
is an isomorphism on $V$. 
\end{Def}

The trivialization change of $E$ via $g(z) \in GL(2)(z)$ changes
$A(z)$ by the following $q$-gauge transformation
\begin{equation}   
\label{gauge tr}
A(z)\mapsto g(qz)A(z)g(z)^{-1}.
\end{equation}
giving $A$ the structure of $(GL(2),q)$-connection. 

\vskip.1in
Let $\mathcal{L}=\text{Span}(s(z))$ in some trivialization on $V$. Then the oper condition \eqref{eq:orbopdef} can be presented as
\begin{equation}
\label{eq:opercondtriv}
s(qz)\wedge A(z)s(z) = \Lambda(z)\,,
\end{equation}
where roots of $\Lambda(z)$ are the oper singularities.

\begin{Def}
A meromorphic Miura $(GL(2),q)$-oper on $\mathbb{P}^1$ is a quadruple $(E,A,\mathcal{L}, \hat{\mathcal{L}})$, where the triple $(E,A,\mathcal{L})$ is the $(GL(2),q)$-oper as above and the line subbundle $\hat{\mathcal{L}}$ is preserved by the $q$-connection $A(z)$.

We require that flags $\mathcal{L}$ and $\hat{\mathcal{L}}$ are in generic relative position.
\end{Def}

In local coordinates, in $\hat{\mathcal{L}}=\text{Span}\begin{pmatrix}
    1\\0
\end{pmatrix}$ the $q$-connection takes the following form
\begin{equation}
A(z)=\left(
\begin{array}{cc}
 g_1(z) & \Lambda (z) \\
 0 & g_2(z) \\
\end{array}
\right)
\end{equation}
for some meromorphic functions $g_{1,2}(z)$ and $\Lambda(z)$. In this gauge $\mathcal{L}=\text{Span}\begin{pmatrix}
    0\\1
\end{pmatrix}$.

\begin{Def}
The Miura $(G,q)$-oper is called \textit{generalized $Z$-twisted} if there is a $q$-gauge transformation $v(z)\in B_+(z)$ such that
\begin{equation}\label{eq:MiuraGen}
A(z)=v(qz)Z(z)v(z)^{-1}
\end{equation}
where $Z(z)\in H(z)$. In local coordinates,
\begin{equation}
\label{eq:qConnDiag}
Z(z) = 
\begin{pmatrix}
\xi_1(z) & 0\\
0 & \xi_2(z) 
\end{pmatrix}\,,
\end{equation}
When $Z(z)\in H\subset H(z)$, we recover the $Z$-twisted $q$-oper.
\end{Def}

In addition, we require that $Z(z)$ is regular semisimple, which means that the locus where $\xi_1(z)=\xi_2(z)$ should be excluded.

\vskip.1in
Now we chose the gauge where the $q$-connection has the form \eqref{eq:qConnDiag}. Let the Miura gauge transformation be as follows
\begin{equation}
v(z)=\left(
\begin{array}{cc}
 Q_+(z)^{-1} & Q_-(z) \\
 0 & Q_+(z) \\
\end{array}
\right)
\end{equation}
where $Q_-(z)\in \mathbb{C}(z)$ such that in this gauge
\begin{equation}\label{eq:sSecGauge}
    s(z)=v(z)\begin{pmatrix}
    0\\1
\end{pmatrix}=\begin{pmatrix}
    Q_-(z)\\Q_+(z)
\end{pmatrix}\,.
\end{equation}
Then the right hand side of the Miura condition for the generalized $Z$-twisted $q$-oper \eqref{eq:MiuraGen} is satisfied provided that
the \textit{generalized $QQ$-system} holds
\begin{equation}\label{eq:LambdaMatching11}
\xi_2(z) Q_+(z)Q_-(q z)-\xi_1(z) Q_-(z) Q_+(q z)=\Lambda(z)\,.
\end{equation}
Notice that this is precisely the quantum Wronskian condition \eqref{eq:opercondtriv} for $s(z)$ given by \eqref{eq:sSecGauge}.
Here the adjective `generalized' refers to the coordinate dependence of the twists. 

This brings us to the following
\begin{Prop}
    There is a one-to-one correspondence between the set of generalized $Z$-twisted Miura $(GL(2),q)$-opers and the set of solutions of the generalized $QQ$-system \eqref{eq:LambdaMatching11}.
\end{Prop}

In the earlier work \cite{KSZ} line bundle $\mathcal{L}$ was trivial on $\mathbb{P}^1\backslash\infty$ so $Q_\pm(z)$ could be chosen as (Taylor) polynomials. In this work, however, we study opers on a $\mathbb{Z}/2\mathbb{Z}$ orbifold of $\mathbb{P}^1$ so we need to consider Laurent polynomials instead which have no poles other than $0,\infty\in \mathbb{P}^1$. We shall discuss these matters momentarily and the orbifold condition will force the components of the section $s(z)$ to belong to a subset of symmetric Laurent polynomials: $Q_\pm(z)\subset\mathbb{C}\left[z+\frac{1}{z}\right]\subset\mathbb{C}[z,z^{-1}]$.

\subsection{Reflection-Invariant \texorpdfstring{$(GL(2),q)$}{(GL(2),q)}-Opers}\label{ref construction}
In order to achieve our goal of recovering Bethe Ansatz equations for the open spin chains, we need to supplement the above $q$-oper construction with the following $\mathbb{Z}_2$ involution $T$ acting on the base $\mathbb{P}^1$ as $z\mapsto z^{-1}$. 

\begin{Def}\label{Def:REflInf2}
    A generalized $Z$-twisted Miura $(GL(2),q)$-oper is called \textit{reflection-invariant} if there exists a $q$-gauge transformation $v(z)\in B_+(z)$ such that $s(z)=v(z)\begin{pmatrix}
        0\\
        1
    \end{pmatrix}$ is invariant under $T$: $s(\frac{1}{z})=s(z)$ and the $q$-connection takes the form \eqref{eq:qConnDiag}.
\end{Def}

In other words, the reflection invariance of the $q$-oper assumes that there exists a gauge in which a section of the line $\mathcal{L}$ is invariant under the reflection of the base local coordinate through the unit circle.

In particular in means that we can represent
\begin{equation}\label{eq:QSymDef}
    Q_+(z)=\delta\prod_{i=1}^k(z-s_i)\left(\frac{1}{z}-s_i\right)
\end{equation}
in this gauge. It is evident that $Q_+(z)\in \mathbb{C}\left[z+\frac{1}{z}\right]$.

\vskip.1in
We need the following lemma which is a consequence of the definition of determinants.
\begin{Lem}\label{Th:DetTransform}
   Let $F$ be any field and let $M\in GL_{n}(F)$. For any $u_{1},\ldots,u_{n}\in F^{n}$ and $1\leq i\leq n$, $u_{1}\land\cdots\land u_{i-1}\land Mu_{i}\land u_{i+1}\land\cdots\land u_{n}$ is equal to
\begin{equation}
\det(M)\left(M^{-1}u_{1}\right)\land\cdots\land \left(M^{-1}u_{i-1}\right)\land u_{i}\land \left(M^{-1}u_{i+1}\right)\land \cdots\land\left(M^{-1}u_{n}\right).
\end{equation}
\end{Lem}

We will see below that the reflection invariance will impose certain transformation properties on the $q$-connection and  $\Lambda(z)$ which controls the oper singularities.

\begin{Thm}\label{Prop:ReflectionInv}
    For a generalized $Z$-twisted reflection-invariant Miura $(GL(2),q)$-oper we have the following condition on the oper singularities
    \begin{equation}\label{eq:LambdaRefl}
     \Lambda\left(\frac{1}{qz}\right)=\Lambda(z).
  \end{equation}
Moreover, in the gauge where line $\mathcal{L}$ is $T$-invariant, we have
    \begin{equation}
    \label{eq:qConnReflection}
    A(z)A\left(\frac{1}{qz}\right)=-\det A \left(\frac{1}{qz}\right)\cdot 1\,,
  \end{equation}
where $1$ is the $2\times 2$ identity matrix.
\end{Thm}
\begin{proof}
The oper condition \eqref{eq:opercondtriv} at the point $\frac{1}{z}$ reads
\begin{equation}
s\left(\frac{q}{z}\right)\land A\left(\frac{1}{z}\right)s\left(\frac{1}{z}\right)=\Lambda\left(\frac{1}{z}\right).
\end{equation}
Since $s$ is invariant under $T$, that is, $s\left(\frac{1}{z}\right)=s\left(z\right)$, we have
\begin{equation}
s\left(\frac{q}{z}\right)\land A\left(\frac{1}{z}\right)s\left(z\right)=\Lambda\left(\frac{1}{z}\right).
\end{equation}
Substituting $qz$ in place of $z$, we have
\begin{equation}
s\left(\frac{1}{z}\right)\land A\left(\frac{1}{qz}\right)s\left(qz\right)=\Lambda\left(\frac{1}{qz}\right).
\end{equation}
By the invariance of $s(z)$ under $T$, we get
\begin{equation}
s\left(z\right)\land A\left(\frac{1}{qz}\right)s\left(qz\right)=\Lambda\left(\frac{1}{qz}\right).
\end{equation}
The above Lemma, together with \eqref{eq:qConnDiag} gives
\begin{equation}
\det A \left(\frac{1}{qz}\right)A\left(\frac{1}{qz}\right)^{-1}s\left(z\right)\land s\left(qz\right)=\Lambda\left(\frac{1}{qz}\right),
\end{equation}
or, equivalently
\begin{equation}\label{eq:opercondreflpoint}
-s\left(qz\right)\land\det A \left(\frac{1}{qz}\right)A\left(\frac{1}{qz}\right)^{-1}s\left(z\right)=\Lambda\left(\frac{1}{qz}\right).
\end{equation}
Now let us prove that $\Lambda(z)=\Lambda(1/qz)$. Indeed, subtracting \eqref{eq:opercondreflpoint} from \eqref{eq:opercondtriv}, we get the following equation
\begin{equation}\label{eq:subtracting}
s(qz)\wedge \left(A(z)+\det A \left(\frac{1}{qz}\right)A\left(\frac{1}{qz}\right)^{-1}\right)s(z) 
= \Lambda(z)-\Lambda\left(\frac{1}{qz}\right).
\end{equation}
We can see that the right hand side in the above equation depends on the oper singularities -- roots of $\Lambda(z)$ whereas the left hand side does not. 
It means that both sides must vanish independently.

In the right hand side, we have a difference of two Laurent polynomials. Thus, at a generic point $z\in\mathbb{P}^1$, the coefficients in front of each power $z^i$ of $\Lambda(z)-\Lambda(1/qz)$ must vanish. This forces $\Lambda(z)\in\mathbb{C}[z+\frac{1}{qz}]$ or $\Lambda(z)=\Lambda(1/qz)$.

For the left hand side of \eqref{eq:subtracting} we get
\[
s(qz)\wedge \left(A(z)+\det A \left(\frac{1}{qz}\right)A\left(\frac{1}{qz}\right)^{-1}\right)s(z)
=0.
\]
Again, we can exploit a genericity argument here. Indeed, if we want the above equation to vanish for an arbitrary section $s(z)$ for generic values of $q$ the operator inside the parentheses must act trivially so we get
\begin{equation}
-\det A \left(\frac{1}{qz}\right)A\left(\frac{1}{qz}\right)^{-1}=A\left(z\right).
\end{equation} 
This completes the proof of the proposition.
\end{proof}

Thus the refection-invariant condition $s(1/z)=s(z)$ can be stated at the level of the $q$-connection as \eqref{eq:qConnReflection}.

\vskip.1in
The above proposition implies in particular that the function $\Lambda(z)$ has the following form
\begin{equation}\label{eq:LambdaNew}
    \Lambda(z)=\gamma\prod_{i=1}^n(z-a_i)\left(\frac{1}{qz}-a_i\right)\,,
\end{equation}
where $\gamma$ is some nonzero constant. The condition \eqref{eq:qConnReflection} takes the following form
\begin{equation}
\label{eq:AreflectionCompts}
   \xi_1\left(\frac{1}{qz}\right)=-\xi_2(z),\quad \xi_2\left(\frac{1}{qz}\right)=-\xi_1(z)\,.
\end{equation}

Notice that $Q_+(z)$ in \eqref{eq:QSymDef} is invariant under $T: z\mapsto \frac{1}{z}$ whereas $\Lambda(z)$ is invariant under $M_q\circ T:z\mapsto \frac{1}{qz}$.
In other words, $Q_+(z)\in\mathbb{C}[z+\frac{1}{z}]$ whereas $\Lambda(z)\in\mathbb{C}[z+\frac{1}{qz}]$.

\subsection{Nondegeneracy}
Since each Laurent polynomial from $\mathbb{C}[z+\frac{1}{z}]$ can be factored as $C\prod_i(z-t^{+}_i)(1/z-t^{-}_i)$, where $C$ is a constant. We assume that no two zeroes of any of the functions that appear in the above theorem belong to the same integral $q$-lattice: $q^\mathbb{Z}t^{+}_i\cap q^\mathbb{Z}t^{-}_j=q^\mathbb{Z}(t^{+}_i)^{\pm 1}\cap q^\mathbb{Z}(t^{+}_j)^{\pm 1}=\varnothing$.
Notice that each such Laurent polynomial can be written as a Taylor polynomial $z^{-l}C\prod_i(z-t^{+}_i)(1-t^{-}_iz)$ divided by $z^l$ for some positive $l$.  

We shall refer to the $QQ$-system of the form \eqref{eq:LambdaMatching11} in the presence of the above nodegeneracy conditions for Laurent polynomials $Q_\pm(z)$ and $\Lambda(z)$ as well as rational functions $\xi_{1,2}(z)$ as nondegenerate $QQ$-system.

\subsection{Bethe Ansatz Equations for Generalized $Z$-Twisted $(GL(2),q)$-Opers}
Using the above nondegeneracy conditions we can formulate the following

\begin{Thm}
Let $Q_+(z)=\prod_{i=1}^k(z-s_i)\left(\frac{1}{z}-s_i\right)$. Then there is a one-to-one correspondence between the set of nondegenerate generalized $Z$-twisted reflection-invariant Miura $(GL(2),q)$-opers and the set of nondegenerate solutions of the following generalized $\mathfrak{gl}_2$ XXZ Bethe equations
\begin{equation}\label{bethe from VW}
-\frac{\xi_1(s_{i})}{\xi_2(q^{-1}s_{i})}\frac{Q_{+}(qs_{i})}{Q_{+}(q^{-1}s_{i})}\frac{\Lambda(q^{-1}s_{i})}{\Lambda(s_{i})}=1\,, \quad i=1,\dots,k
\end{equation}
\end{Thm} 

\begin{Rem}
    Notice that without imposing the reflection invariance condition $Q_\pm(z)$ and $\Lambda(z)$ can be considered as arbitrary not necessarily symmetric Laurent polynomials in $\mathbb{C}[z,z^{-1}]$ and the above theorem would also hold.
\end{Rem}

The statement is proved analogously in the presence of the above nondegeneracy conditions to Theorem 3.6 of \cite{KSZ} where one evaluates both sides of \eqref{eq:LambdaMatching11} at $z=s_i$. 

Note that instead of evaluating the $QQ$-equation at $s_i$s one can do so at $s_i^{-1}$. This procedure will result in dual Bethe equations which are equivalent to \eqref{bethe from VW}. This is the manifestation of the so-called \textit{space-reflection} symmetry \cite{sklyaninboundary}. In the right picture of \figref{fig:foldedchain} this choice would correspond to composing intertwiners of $V_i(a_i)$ with $V_{i-1}(a_{i-1}), V_{i-2}(a_{i-2}),\dots$ instead and reflecting off of the left boundary first.

\subsection{Singularities and Asymptotics of the $q$-Connection.}
Bethe equations \eqref{bethe from VW} describe the space of nondegenerate reflection-invariant $(GL(2),q)$-opers. At the moment any rational functions $\xi_{1,2}(z)$ that satisfy \eqref{eq:AreflectionCompts} are allowed. Next we will impose asymptotic conditions at $0$ and $\infty$ on the $q$-connection $A(z)$ in this diagonal gauge which will further restrict $\xi_{1,2}(z)$.  

\subsubsection{Interlude: Differential Connections}
The choice of these asymptotic conditions is motivated by similar analysis for differential connections $a(z)\in \text{End}(E)\otimes K$, where $K$ is the canonical bundle. For a differential connection, one typically assumes mild (tame) singularity at $z=0$ so that $a(z)=a_{-1}/z+a_0+O(z)$, where $a_{-1},a_0\in\mathfrak{g}$ for some Lie algebra $\mathfrak{g}$, i.e. $\mathfrak{gl}_2$. Then, if we also know the asymptotic behavior of the connection at infinity, then by Riemann-Roch theorem we can obtain some information about zeros and poles in $\mathbb{P}^1\backslash\{0,\infty\}$.

Let us first consider such a prototypical example. Assume that $a(z)=a_0dz+O(z)$ as $z\to 0$ and that in a local trivialization $a(z)=\begin{pmatrix}
    x_1(z) &0\\
    0& x_2(z)
\end{pmatrix}dz$
for some rational functions $x_{1,2}(z)$ satisfying a similar condition to \eqref{eq:AreflectionCompts}, i.e. $x_1(z^{-1})=-x_2(z)$. In this case, we get $a(z)=a'_0dz+O(z^{-1})$ as $z\to \infty$. However, in the punctured neighborhood of $\infty\in\mathbb{P}^1$ the local coordinate is $w=\frac{1}{z}$ thus $dz = -\frac{dw}{w^2}$ so the connection has a double pole. Since the degree of any canonical divisor on $\mathbb{P}^1$ should be equal to $-2$ we are allowed to add equal number of zeros and poles. For instance, 
\begin{equation}\label{eq:case1}
    x_1(z)=\frac{(z-a)(z-b)}{(z-c)(z-d)}
\end{equation}
for nonzero complex $a,b,c$, and $d$ will work.

On the other hand, if we assume that the connection has a simple pole at the origin $a(z)=\frac{a_{-1}}{z}dz+a_0dz+O(z)$ and by the reflection condition $x_1(z^{-1})=-x_2(z)$ it reads $a(z)=a'_{1}zdz+a'_0dz+O(z^{-1})$ as $z\to\infty$. After transforming to the local coordinate $w$ at infinity we see that $a(z)$ has a triple pole. In order to satisfy the Riemann-Roch theorem we therefore need to add two simple zeros. In other words,
\begin{equation}\label{eq:case2}
    x_1(z)=\frac{(z-f)(z-h)}{z}
\end{equation}
for nonzero complex $f$ and $h$ will be sufficient. Note also that the product of the rational functions in \eqref{eq:case1} and \eqref{eq:case2} also satisfies the latter asymptotic conditions.

\subsubsection{Constant Asymptotics of $A(z)$}\label{VWasymp}
Now let us turn back to our $q$-oper connection $A(z)$. First, let us assume that it asymptotes to a constant matrix which we can separate into the $GL(1)$ and $SL(2)$ parts as $A(z)=b\cdot a_0^{\frac{\check{\alpha}}{2}}$ as $z\to 0$ where $\check{\alpha}=\text{diag}(1,-1)$ is the Cartan Chevalley generator of $\mathfrak{sl}_2$. From the reflection conditions \eqref{eq:AreflectionCompts} we then get that $b=-a_0^{1/2}$ and
\begin{equation}
    A(z)\sim - a_0^{1/2}\begin{pmatrix}
        a_0^{1/2}&0\\
        0&a_0^{-1/2}
    \end{pmatrix},\quad z\to 0,\qquad
        A(z)\sim a_0^{1/2}\begin{pmatrix}
        a_0^{-1/2}&0\\
        0&a_0^{1/2}
    \end{pmatrix},\quad z\to \infty,
\end{equation}
Next, the $q$-connection is expected to have poles at $z=\pm q^{-1/2}$ which are the two fixed points of the involution $z\mapsto \frac{1}{qz}$. In order to maintain the above asymptotics one needs to add two more zeros $\mu$ and $\tilde\mu$ (cf. \eqref{eq:case1}). 

Ultimately we arrive at the following explicit formulae for the components of $Z(z)$
\begin{equation}\label{eq:Aconndiagbc}
\xi_1(z)=q\frac{(z-\mu)(z-\tilde\mu)}{qz^{2}-1}, \quad \xi_2(z)=\frac{(q\mu z-1)(q\tilde \mu z-1)}{qz^{2}-1}.
\end{equation}
Here $a_0=q\mu\tilde{\mu}$.

\subsubsection{Simple Pole at Zero of $A(z)$}\label{simple pole at zero}
Assume now that 
\begin{equation}
    A(z)=-\frac{c}{z}d_{-1}^{\frac{\check{\alpha}}{2}} + a_0 + O(z)
\end{equation}
as $z\to 0$. We want to incorporate the constant asymptotics from \eqref{eq:Aconndiagbc} into this formula as well. 
From \eqref{eq:AreflectionCompts} we get that $c^2 = d_{-1}$ and 
\begin{equation}
    A(z)=d_{-1}^{1/2}\cdot d_{-1}^{-\frac{\check{\alpha}}{2}}\,z + a'_0 + O(z)
\end{equation}
as $z\to \infty$. This results in the following expressions for $\xi_{1,2}(z)$ (cf. \eqref{eq:case2})
\begin{align}\label{eq:xi1full}
\xi_1(z)&=q^2\frac{(z-\mu)(z-\tilde{\mu})}{qz^{2}-1}\cdot\frac{(z+b)(z+\tilde{b})}{z}\,,\cr
\xi_2(z)&=\frac{(q z\mu-1)(q z\tilde{\mu}-1)}{qz^{2}-1}\cdot\frac{(qzb+1)(qz\tilde{b}+1)}{z}\,.
\end{align}
Similarly to \eqref{eq:Aconndiagbc} in this calculation and $d_{-1}=q\mu\tilde{\mu}b\tilde{b}$.

\subsection{Bethe Ansatz Equations}
We are now ready to write down Bethe Ansatz equations \eqref{bethe from VW}  for generalized $Z$-twisted Miura $(GL(2),q)$-opers explicitly in the presence of the asymptotic conditions \eqref{eq:Aconndiagbc} and \eqref{eq:xi1full}. 

\subsubsection{Constant Asymptotics of $A(z)$}
Let $Q_+(z)$ be as in \eqref{eq:QSymDef}, $\Lambda(z)$ as in \eqref{eq:LambdaNew} and the $A(z)$ is given by components from \eqref{eq:Aconndiagbc}. Then the Bethe equations  take the following explicit form
\begin{equation}\label{eq:Bethesl2diag}
    q\frac{s_i^2-q }{q s_i^2-1}\frac{\left(s_i-\mu\right)
   \left(s_i-\tilde \mu\right)}{\left(\mu s_i-1\right) \left(\tilde \mu  s_i-1\right)}\prod_{j=1}^k\frac{\left(q
   s_i-s_j\right) \left(q s_i s_j-1\right)}{\left(s_i-q s_j\right) \left(s_i
   s_j-q\right)}\prod_{l=1}^n\frac{\left(a_l s_i-1\right) \left(a_l
   q-s_i\right)}{\left(a_l-s_i\right) \left(a_l q s_i-1\right)}=-1\,,
\end{equation}
where $i=1,\dots,k$.

\subsubsection{Simple Pole at Zero of $A(z)$}
For the asymptotics from \eqref{eq:xi1full} we get the following Bethe equations
\begin{align}\label{eq:Bethesl2full}
    & \frac{s_i^2-q }{q s_i^2-1}\frac{\left(s_i-\mu\right)
   \left(s_i-\tilde \mu\right)}{\left(\mu s_i-1\right) \left(\tilde \mu  s_i-1\right)}\frac{\left(s_i+b\right)
   (s_i+\tilde b)}{\left(b s_i+1\right) (\tilde b  s_i+1)}\cr
   &\cdot\prod_{j=1}^k\frac{\left(q
   s_i-s_j\right) \left(q s_i s_j-1\right)}{\left(s_i-q s_j\right) \left(s_i
   s_j-q\right)}\prod_{l=1}^n\frac{\left(a_l s_i-1\right) \left(a_l
   q-s_i\right)}{\left(a_l-s_i\right) \left(a_l q s_i-1\right)}=-1\,,
\end{align}
where $i=1,\dots,k$.

\subsubsection{Bethe Ansatz for \texorpdfstring{$U_q(\hat{\mathfrak{sl}}_2)$}{U_q(sl_2)} Open Spin Chains}
Bethe equations \eqref{eq:Bethesl2diag} first appeared in the seminal work of Sklyanin \cites{sklyaninboundary} where he used the quantum inverse scattering method. Later in \cites{Vlaar:2020jww} the same equations were reproduced from the representation theory of quantum affine algebra $U_q(\hat{\mathfrak{sl}}_2)$. These equations describe the XXZ spin chain on $n$ sites in the sector with $k$ excitations with diagonal boundary conditions determined by parameters $\mu,\tilde\mu$. The Hilbert space of the corresponding quantum mechanical problem is the evaluation tensor module $\mathbb{C}^2(a_1)\otimes\dots\otimes\mathbb{C}^2(a_n)$ where $\mathbb{C}^2$ is the fundamental representation of $\mathfrak{sl}_2$. The diagonal boundary conditions appear in a particular solution of the reflection equation.

In our geometric setting, \eqref{eq:Bethesl2diag} are realized from $q$-connections that approach a constant at $z=0$.

\vskip.1in

Yang, Nepomechie, and Zhang \cite{Yang:2005ce} derived equations \eqref{eq:Bethesl2full} using the most general solution of the reflection equation which involves one additional parameter for the left and right boundary of the spin chain $b$ and $\tilde b$. We provide the exact match of notations in \secref{Sec:Connect}.

Notice that this is not a direct limit like $b,\tilde b\to 0$ under which \eqref{eq:Bethesl2full} turns into \eqref{eq:Bethesl2diag}. One needs to additionally scale variables $\mu$ and $\tilde{\mu}$  to make it work.

\vskip.1in
In our geometric setting, \eqref{eq:Bethesl2full} are realized from $q$-connections which have a simple pole at $z=0$. Therefore, generalized $Z$-twisted reflection-invariant $(GL(2),q)$-opers whose connections have at most \textit{tame} singularities are described by $\mathfrak{sl}_2$ XXZ Bethe equations with generic open boundary conditions.

\subsection{Attenuation}
In the above construction we had $a_0=q\mu\tilde\mu$ for the constant asymptotic conditions of $A(z)$. In other words, the $q$-connection approaches diagonal matrix $Z=\text{diag}(1,a_0)$ at infinity. In other words, this limit reproduces the $q$-connection behavior for closed chains. This phenomenon is called \textit{attenuation} in the spin chain literature \cite{lukyanenko2014boundaries}.

\section{Reflection-invariant \texorpdfstring{$(GL(N),q)$}{(GL(N),q)}-Opers}\label{ref inv gln opers}
Now let us study the higher rank setup.

\subsection{\texorpdfstring{$(GL(N),q)$}{(GL(N),q)}-Opers}
Let $q\in\mathbb{C}^\times$. Consider the holomorphic vector bundle $E$ of rank $N$ over $\mathbb{P}^1$ that is equivariant with respect to the scaling of the base $M_q:\mathbb{P}^{1}\to \mathbb{P}^{1}$ that sends $z\mapsto qz$. Let $E^q$ denote the pullback of $E$ under $M_q$.

\begin{Def}\label{Def:qOper}
    A meromorphic $(GL(N),q)$-oper on $\mathbb{P}^1$ is a triple $(E,A,\mathcal{L}_\bullet)$ where for an open Zariski dense subset $U\subset\mathbb{P}^1$ and $V=U\cap M_q^{-1}(U)$ the $q$-connection $A\in\text{Hom}_{\mathcal{O}(V)}(E,E^q)$ satisfies the following conditions on the complete flag $\mathcal{L}_\bullet$ of subbundles $\mathcal{L}_{1}\subset \dots\subset\mathcal{L}_{N-1}\subset E$

    i) $A\cdot \mathcal{L}_i\subset \mathcal{L}_{i+1}$

    ii) The induced maps $\bar{A_i}:\mathcal{L}_{i+1}/\mathcal{L}_{i}\to \mathcal{L}_{i}/\mathcal{L}_{i-1}$ are isomorphisms.
\end{Def}

Changing the trivialization of $E$ via $g(z)\in GL(N)(z)$ is given by the following $q$-gauge transformation
\begin{equation}
    A(z)\mapsto g(qz)A(z)g(z)^{-1} \,.
\end{equation}

\begin{Def}
    We say that the oper is \textit{generalized $Z$-twisted} if there is a gauge change $g(z)\in GL(N)(z)$ such that $g(qz)A(z)g(z)^{-1}=Z(z)\in H(z)$ 
\end{Def}

\begin{Rem}
    At this point, we observe a deviation from \cite{Frenkel:2020} and previous work on $q$-opers and Bethe Ansatz. Indeed, a generalized $Z$-twisted $q$-connection can be diagonalized into a connection that takes values in loops valued in $H$, not just in constant loops.
\end{Rem}

\begin{Def}
    A Miura $(GL(N),q)$-oper is a quadruple $(E,A,\mathcal{L}_\bullet,\hat{\mathcal{L}}_\bullet)$ where $(E,A,\mathcal{L}_\bullet)$ is a $(GL(N),q)$-oper and the complete flag $\hat{\mathcal{L}}_\bullet$ of subbundles is preserved by the $q$-connection.
\end{Def}

\begin{Def}
The Miura $(GL(N),q)$-oper is called \textit{generalized $Z$-twisted} if there is a $q$-gauge transformation $v(z)\in B_+(z)$ such that
\begin{equation}\label{eq:MiuraGenN}
A(z)=v(qz)Z(z)v(z)^{-1}
\end{equation}
where $Z(z)\in H(z)$.
\end{Def}

Let as denote $Z(z)=\text{diag}(\xi_1,\dots, \xi_N)(z)$ where $\xi_i(z)\in\mathbb{C}(z)$.

The following statement generalizes a similar result for $Z$-twisted $q$-opers from earlier work.
\begin{Prop}
There are exactly $N!$ Miura $(GL(N),q)$-opers for a given generalized $Z$-twisted $(GL(N),q)$-oper provided that $Z(z)$ is regular semisimple.
\end{Prop}

At certain loci on the base of $E$ this condition will fail and it  will become important for the construction to follow. In particular, $\xi_i(z)$ may have poles and zeros.

\begin{Def}
    We say that a (Miura) $(GL(N),q)$-oper has regular singularities defined by a collection of Laurent polynomials $\Lambda_i(z)$ for $i=1,\dots, N-1$ if the induced maps $\bar{A}_i$ in the the Definition \ref{Def:qOper} are isomorphisms away from the roots of these polynomials.
\end{Def}

\begin{Rem}
    Another deviation from \cite{Frenkel:2020} -- all components of sections of subbundles of $E$ as well as $\Lambda_i(z)$ need to be considered as Laurent polynomials. In what follows, unless specified, we will be dealing with Laurent polynomials.
\end{Rem}

The generalized $Z$-twisted condition of the Miura $q$-oper implies that in the gauge when $A$ is equivalent to $Z(z)\in H(z)$ the flag $\hat{\mathcal{L}}_\bullet$ is formed by the standard basis $e_1, \dots, e_{N}$. The relative position between flags $\hat{\mathcal{L}}$ and $\hat{\mathcal{L}}_\bullet$ is generic on $V\subset\mathbb{P}^1$. 

Let $s(z)$ to be the local section that generates $\mathcal{L}_1$ according to the Definition \ref{Def:qOper}. The regular singularity condition implies that $q$-Wronskians, namely the determinants
\begin{align}\label{qD}
&\mathcal{D}_k(s)=e_1\wedge\dots\wedge{e_{N-k}}\\
&\wedge s(q^{k-1}z)\wedge Z(q^{k-2}z) s(q^{k-2}z)\wedge Z(q^{k-2}z)Z(q^{k-3}z) s(q^{k-3}z)\wedge\dots\wedge Z(q^{k-2}z)\cdots Z(z)s(z)\,\notag
\end{align}
have a subset of zeroes, which come from the following Laurent polynomials 
\begin{eqnarray}\label{eq:WPDefs}
W_k(s)=P_1(z) \cdot P_2(q^2z)\cdots P_{k}(q^{k-1}z),  \qquad 
P_k(z)=\Lambda_{N-1}\Lambda_{N-2}\cdots\Lambda_{N-k}(z)\,.
\end{eqnarray}  

Another subset of zeros of $\mathcal{D}_k(s)$ corresponds to points on the base where the two flags are not in a generic relative position, i.e., when one of the vectors $Z^{k-j}s(q^{j-1})$ becomes parallel to a vector from the standard basis. Therefore, the Miura condition is equivalent to the existence of nonzero constants $\alpha_k$ and Laurent polynomials $\mathcal{V}_k(z)$
for which the following equations hold
\begin{equation}\label{eq:Dminor}
    \mathcal{D}_k(s)= \alpha_{k} W_{k} \mathcal{V}_{k}(s)\,,
\end{equation}
In components, the above functions can be represented as determinants of the matrix $M_{N-k+1,\dots, N}$
\begin{equation}\label{eq:MijMatrix}
    \left(M_{N-k+1,\dots, N}\right)_{i,j}=\xi_{N-k+i}(q^{k-1}z)\cdots \xi_{N-k+i}(q^{k-j}z)\cdot s_{N-k+i}(q^{k-j}z)\,,\quad i,j=1,\dots,k
\end{equation}
The boundary conditions for \eqref{eq:Dminor} are 
\begin{equation}\label{eq:cDGeneric}
    \mathcal{D}_{N}(s)=W_{N}(s)\,,\qquad  \cV_{N}(s)=1\,, \qquad \mathcal{D}_0=e_1\wedge\dots\wedge e_{N}=1\,.
\end{equation}

However, for the purposes of this exposition, it will be sufficient to keep only $\Lambda_N(z)=\Lambda(z)$ nontrivial whereas $\Lambda_i(z)=1$ for $i=1,\dots, N-1$. It means that \eqref{eq:Dminor} can be viewed as
\begin{equation}
    \mathcal{D}_{k}(s)(z)=\mathcal{V}_{k}(s)(z)\,, \quad k=1,\dots,N-1\,, \qquad \mathcal{V}_{N}(s)(z)=\Lambda(z)\,.
\end{equation}
Generalizations for the complete set of singularities are straightforward.
\begin{Rem}\label{determinant}
The first determinant $\mathcal{D}_1(s)(z)=e_1\wedge\cdots\wedge e_{N-1}\wedge s(z)$ is then equal to $\mathcal{V}_1(z)$, which will later be identified with $Q_1^+(z)$, if we put $\alpha_1=1$.
\end{Rem}

\subsubsection*{Example: \texorpdfstring{$(GL(4),q)$}{(GL(4),q)}-Oper}
For the $(GL(4),q)$-oper the matrix $M_{1,2,3,4}$ reads
$$
\left(
\begin{array}{cccc}
 s_1\left(q^3 z\right) & s_1\left(q^2 z\right) \xi _1\left(q^2 z\right) & s_1(q z) \xi _1(q
   z) \xi _1\left(q^2 z\right) & s_1(z) \xi _1(z) \xi _1(q z) \xi _1\left(q^2 z\right) \\
 s_2\left(q^3 z\right) & s_2\left(q^2 z\right) \xi _2\left(q^2 z\right) & s_2(q z) \xi _2(q
   z) \xi _2\left(q^2 z\right) & s_2(z) \xi _2(z) \xi _2(q z) \xi _2\left(q^2 z\right) \\
 s_3\left(q^3 z\right) & s_3\left(q^2 z\right) \xi _3\left(q^2 z\right) & s_3(q z) \xi _3(q
   z) \xi _3\left(q^2 z\right) & s_3(z) \xi _3(z) \xi _3(q z) \xi _3\left(q^2 z\right) \\
 s_4\left(q^3 z\right) & s_4\left(q^2 z\right) \xi _4\left(q^2 z\right) & s_4(q z) \xi _4(q
   z) \xi _4\left(q^2 z\right) & s_4(z) \xi _4(z) \xi _4(q z) \xi_4\left(q^2 z\right) \\
\end{array}
\right)
$$
Notice that all columns are sections of $E^{q^4}$. One then has $\mathcal{D}_4(s)=\det M_{1,2,3,4}(s)$. At the same time, matrix $M_{3,4}$ for $(GL(4),q)$-opers read

\begin{equation}
    \left(
\begin{array}{cc}
 s_3(q z) & s_3(z) \xi _3(z) \\
 s_4(q z) & s_4(z) \xi _4(z) \\
\end{array}
\right)
\end{equation}
whose determinant is equal to $\mathcal{V}_2(z)$ (later to be identified with $Q_2^+(z)$).

\subsection{The \texorpdfstring{$QQ$}{QQ}-System}
As it was shown in \cites{KSZ,Koroteev:2020tv} all $Q^\pm$ functions can be understood as certain (generalized) minors of the matrix $M_{1,\dots, N}$. In the present setup we need to update that construction to iterated products of $q$-shifted twist elements. 

The minors in question obey the famous combinatorial Jacobi/Lewis Carroll identity which for our matrices $M$ can be written as
\begin{equation}\label{eq:LewisCarroll}
    \det M^1_1 \det M^2_N - \det M^1_N \det M^2_1 = \det M^{1,2}_{1,N} \det M\,,
\end{equation}
where $M^a_b$ denotes a submatrix of $M$ which is obtained from $M$ by removing $a$-th row and $b$-th column. 

We can see that if remove the first row and the first column of $M_{N-k+1,\dots, N}$ we will get the matrix $M_{N-k+2,\dots, N}$ whose $i$th row is multiplied by $\xi_i(q^{k-2} z)$. Similar findings can be made about other submatrices from \eqref{eq:LewisCarroll}. 

Let us denote
\begin{equation}
    Q^+_k(z)=\det M_{N-k+1,\dots, N}(z)\,,\qquad  Q^-_k(z)=\det M_{N-k,N-k+2\dots, N}(z)\,.
\end{equation}
It is easy to see that
\begin{equation}
    (M_{N-k,\dots, N})^1_1(z)=\xi_{N-k+1}\cdots \xi_{N}(q^{k-2})\cdot Q^+_k(z)\,,\qquad (M_{N-k,\dots, N})^2_{k+1}(z)=Q^-_k(qz)\,.
\end{equation}
as well as
\begin{equation}
    (M_{N-k,\dots, N})^1_{k+1}(z)=Q^+_k(qz)\,,\qquad (M_{N-k,\dots, N})^2_1(z)=\xi_{N-k}\cdot \xi_{N-k+2}\cdots \xi_{N}(q^{k-2}z)\cdot Q^-_k(z)\,.
\end{equation}
At the same time
\begin{equation}
   (M_{N-k,\dots, N})^{12}_{1\,k+1}(z)=\xi_{N-k+2}\cdots \xi_{N}(q^{k-2}z)\cdot Q^+_{k-1}(qz)\,.
\end{equation}

This brings us to the following
\begin{Thm}\label{Th:QQinj}
  There is a one-to-one correspondence between the set of
  nondegenerate generalized $Z$-twisted Miura $(GL(N),q)$-opers and the set of nondegenerate Laurent polynomial solutions of the $QQ$-system, 
\begin{equation}\label{eq:QQAtype}
\xi_{N-k+1}(q^{k-1}z) Q^+_k(z) Q^-_k(qz) - \xi_{N-k}(q^{k-1}z) Q^+_k(qz) Q^-_k( z) =  Q^+_{k-1}(qz)Q^+_{k+1}(z)\,, 
\end{equation}
where $k = 1,\dots, N-1$ and $Q^+_{N}(z)=\Lambda(z)$.
\end{Thm}

Note that under this identification, the last component of the section vector $s(z)$ is $Q_1^+(z)$ whereas the penultimate component is $Q_1^-(z)$. The other components can be obtained from $Q_1^\pm(z)$ via so-called B\"acklund-type transformations \cite{Frenkel:2020,Koroteev:2020tv}.

\subsection{Nondegeneracy}
Each Laurent polynomial can be factored as $C\prod_i(z-t^{+}_i)\prod_j(1/z-t^{-}_j)$, where $C$ is a constant. We assume that no two zeroes of any of the functions that appear in the above theorem belong the the same integral $q$-lattice: $q^\mathbb{Z}t^{+}_i\cap q^\mathbb{Z}t^{-}_j=q^\mathbb{Z}(t^{+}_i)^{\pm 1}\cap q^\mathbb{Z}(t^{+}_j)^{\pm 1}=\varnothing$.

Notice that each Laurent polynomial can be written as a Taylor polynomial $z^{-l}C\prod_i(z-t^{+}_i)\prod_j(1-t^{-}_jz)$ divided by $z^l$ for some positive $l$.  

Let $\xi_{k}(z)=\sum_{m=N_{k}}^{\infty}\zeta_{k,m}z^{m}$ be the Laurent expansion of $\xi_{k}(z)$ around $0$, where $N_{k}\in\mathbb{Z}$ and $\zeta_{k,N_{k}}\neq 0$. From now on, we also assume that 
\begin{equation}\label{nondegeneracy}
\frac{\zeta_{k+1,\text{min}(N_{k+1},N_{k})}}{\zeta_{k,\text{min}(N_{k+1},N_{k})}}\notin q^{\mathbb{Z}}\quad \text{or} \quad\frac{\zeta_{k,\text{min}(N_{k+1},N_{k})}}{\zeta_{k+1,\text{min}(N_{k+1},N_{k})}}\notin q^{\mathbb{Z}}, \quad 1\leq k\leq N-1.
\end{equation}
\begin{Rem}
    The condition \eqref{nondegeneracy} becomes equivalent to the nondegeneracy in \cite{Frenkel:2020} for constant twists, that is, when $\zeta_{k}\in\mathbb{C}^{\times}$ for $1\leq k\leq N$.
\end{Rem}
\begin{Def}
 Upon such nondegeneracy conditions, the system of functional $q$-difference equations \eqref{eq:QQAtype} is called nondegenerate $GL(N)$ $QQ$-system.
\end{Def}

\subsection{Bethe Equations}
The correspondence between the $QQ$-system and the Bethe Ansatz equations follows, as in the story for the $Z$-twisted $q$-opers, after evaluating both sides of \eqref{eq:QQAtype} before and after a $q$-shift at zeros $s_{k,i}$ of $Q^+_k$.

\begin{Thm}\label{QQtoBethe} For $1\leq k\leq N$, let
\[
Q_{k}^{+}(z)=
\prod_{j=1}^{p_{k}}\left(z-s_{k,j}\right)\left(\frac{1}{q^{k-1}z}-s_{k,j}\right).
\]
Then the solutions of the nondegenerate generalized $GL(N)$ $QQ$-system are in one-to-one correspondence with the nondegenerate solutions of the following generalized open Bethe Ansatz equations
\begin{equation}\label{eq:bethe}
 -1= \frac{\xi_{N-k+1}(q^{k-2}s_{k,i,})}{\xi_{N-k}(q^{k-1}s_{k,i})} \frac{Q^+_{k-1}(qs_{k,i})Q^+_k(q^{-1}s_{k,i})Q^+_{k+1}(s_{k,i})}{  Q^+_{k-1}(s_{k,i})Q^+_k(qs_{k,i})Q^+_{k+1}(q^{-1}s_{k,i})},
\end{equation}
where $k=1,\dots, N-1$.
\end{Thm}

The exact form of the above \textit{generalized open Bethe Ansatz equations} depends on the functions $Q^\pm_i$ and $\xi_i$. If $Q^\pm_i$ are polynomials and $\xi_i$ are nonzero constants then the Bethe equations are given in \cite{Koroteev:2020tv}. In the next section, we shall impose an additional condition on the section $s(z)$ of the line $\mathcal{L}_1$ which will force $Q^\pm_k$ to be Laurent polynomials.

\begin{proof}
    It remains to show that starting from a solution of Bethe equations \eqref{eq:bethe}, we can solve the $QQ$-system \eqref{eq:QQAtype}. The proof of this fact that we give here is an adaptation of the proof in \cite[Theorem 6.4]{Frenkel:2020}. Let us first rewrite the QQ-system \eqref{eq:QQAtype} in the following way:
    \begin{equation}\label{eq:QQusingphi}
        \xi_{N-k+1}(q^{k-1}z) \phi_k(qz) - \xi_{N-k}(q^{k-1}z) \phi_k(z) =  \frac{Q^+_{k-1}(qz)Q^+_{k+1}(z)}{Q^+_{k}(z)Q^+_{k}(qz)}\,, 
    \end{equation}
    where $\phi_{k}(z)=\frac{Q^-_{k}(z)}{Q^+_{k}(z)}$. Note that the RHS of \eqref{eq:QQusingphi} equals
    \[
   \prod_{j=1}^{p_{k-1}}\frac{\left(qz-s_{k-1,j}\right)\left(1-q^{k-1}zs_{k-1,j}\right)}{q^{k-1}z}\cdot\prod_{j=1}^{p_{k}}\frac{q^{2k-1}z^2}{\left(z-s_{k,j}\right)\left(qz-s_{k,j}\right)\left(1-q^{k-1}zs_{k,j}\right)\left(1-q^{k}zs_{k,j}\right)}
   \]
    \[
    \cdot\prod_{j=1}^{p_{k+1}}\frac{\left(z-s_{k+1,j}\right)\left(1-q^{k}zs_{k+1,j}\right)}{q^{k}z}\,.
    \]
    Multiplying both sides of \eqref{eq:QQusingphi} by $q^{(k-1)p_{k-1}+kp_{k+1}}z^{p_{k-1}+p_{k+1}}\prod_{j=1}^{p_{k}}\left(1-q^{k-1}zs_{k,j}\right)\left(1-q^{k}zs_{k,j}\right)$, we get
\begin{align}\label{qdiffeqn}
        &     q^{(k-1)p_{k-1}+kp_{k+1}}z^{p_{k-1}+p_{k+1}}\Big(\prod_{j=1}^{p_{k}}\left(1-q^{k-1}zs_{k,j}\right) \xi_{N-k+1}(q^{k-1}z) \psi_k(qz) \cr
        &- \prod_{j=1}^{p_{k}}\left(1-q^{k}zs_{k,j}\right)\xi_{N-k}(q^{k-1}z) \psi_k(z) \Big)\cr
    &=  
       \frac{R_{k}(z)}{\prod_{j=1}^{p_{k}}\left(z-s_{k,j}\right)\left(qz-s_{k,j}\right)}\,, 
       \end{align}
    where $\psi_{k}(z)=\phi_{k}(z)\prod_{j=1}^{p_{k}}\left(1-q^{k-1}zs_{k,j}\right)$ and $R_{k}(z)$ is a polynomial with no root equal to $s_{k,j}$ or $q^{-1}s_{k,j}$.

    Using partial fraction decomposition, let us write
    \[
    \frac{R_{k}(z)}{\prod_{j=1}^{p_{k}}\left(z-s_{k,j}\right)\left(qz-s_{k,j}\right)}
    =
    h_{k}(z)+\sum_{j=1}^{p_{k}}\frac{b_{j}}{z-s_{k,j}}+\sum_{j=1}^{p_{k}}\frac{c_{j}}{qz-s_{k,j}},
    \]
    where $h_{k}(z)$ is a polynomial. Let us write $\psi_{k}(z)$ as
    \[
    \psi_{k}(z)=\widetilde{\psi}_{k}(z)+\sum_{j=1}^{p_{k}}\frac{d_{j}}{z-s_{k,j}}.
    \]
    The $d_{j}$s are uniquely determined by requiring the residues on both sides of \eqref{qdiffeqn} to be the same, which is exactly the Bethe equations. Let us set 
    \[
    \widetilde{\psi}_{k}(z)=\sum_{m\geq 0}r_{m}z^{m}
    \]
    and 
    \[
    h_{k}(z)=\sum_{m\geq 0}s_{m}z^{m}.
    \]
    It remains to solve the following $q$-difference equation,
    \begin{equation}
        A_{k}(z)\widetilde{\psi}_{k}(qz)-B_{k}(z)\widetilde{\psi}_{k}(z)=h_{k}(z),
    \end{equation}
where 
\[
A_{k}(z)=q^{(k-1)p_{k-1}+kp_{k+1}}z^{p_{k-1}+p_{k+1}}\prod_{j=1}^{p_{k}}\left(1-q^{k-1}zs_{k,j}\right) \xi_{N-k+1}(q^{k-1}z)
\]
and 
\[
B_{k}(z)=q^{(k-1)p_{k-1}+kp_{k+1}}z^{p_{k-1}+p_{k+1}}\prod_{j=1}^{p_{k}}\left(1-q^{k}zs_{k,j}\right)\xi_{N-k}(q^{k-1}z).
\]
Let $A_{k}(z)=\sum_{m=N_{1}}a_{m}z^{m}$ and $B_{k}(z)=\sum_{m=N_{2}}b_{m}z^{m}$ be the Laurent expansion of $A_{k}(z)$ and $B_{k}(z)$ around $z=0$, where $N_{1}, N_{2}\in\mathbb{Z}$ and $a_{N_{1}}\neq 0$ and $b_{N_{2}}\neq 0$. Then comparing the coefficients of the powers of $z$ in \eqref{qdiffeqn} gives
\[
s_m=\sum_{j=\text{min}(N_{1},N_{2})}^{m}\left(a_{j}q^{m-j}-b_{j}\right)r_{m-j}, \quad m\geq 0.
\]
The above equations can now be solved since $a_{\text{min}(N_{1},N_{2})}q^{m-\text{min}(N_{1},N_{2})}-b_{\text{min}(N_{1},N_{2})}\neq 0$, this follows from our assumption \eqref{nondegeneracy}.  We note here that $a_{N_1}$ and $b_{N_2}$ depend only on $\xi_{N-k+1}$ and $\xi_{N-k}$ and not on Bethe roots.

\end{proof}
\begin{Rem}
    If in the above theorem, instead of evaluating at $s_{k,i}$ and $q^{-1}s_{k,i}$, we evaluate at $\dfrac{1}{q^{k-1}s_{k,i}}$ and $\dfrac{1}{q^{k}s_{k,i}}$ we obtain the following Bethe equations which are equivalent to the original ones
    \begin{equation}
 -1= \frac{\xi_{N-k+1}\left(\dfrac{1}{qs_{k,i}}\right)}{\xi_{N-k}\left(\dfrac{1}{s_{k,i}}\right)} \frac{Q^+_{k-1}\left(\dfrac{1}{q^{k-2}s_{k,i}}\right)Q^+_k\left(\dfrac{1}{q^{k}s_{k,i}}\right)Q^+_{k+1}\left(\dfrac{1}{q^{k-1}s_{k,i}}\right)}{  Q^+_{k-1}\left(\dfrac{1}{q^{k-1}s_{k,i}}\right)Q^+_k\left(\dfrac{1}{q^{k-2}s_{k,i}}\right)Q^+_{k+1}\left(\dfrac{1}{q^{k}s_{k,i}}\right)}.
\end{equation}

\end{Rem}

\subsection{Reflection-Invariant Conditions}
The following definition is a straightforward generalization of Def. \ref{Def:REflInf2}.
\begin{Def}
    A generalized $Z$-twisted Miura $(GL(N),q)$-oper is called \textit{reflection-invariant} if there exists a $q$-gauge transformation $v(z)\in B_+(z)\subset GL(N)(z)$ such that $s(z)=v(z) (0, 0,\dots, 0, 1)^T$ is invariant under the reflection through the unit circle $T$: $s(\frac{1}{z})=s(z)$ and the $q$-connection takes the form $Z(z)\in H(z)$.
\end{Def}

The following is a straightforward generalization of the proof of Proposition \ref{Prop:ReflectionInv}.

\begin{Prop}\label{Th:PropInvA}
For the reflection-invariant generalized $Z$-twisted Miura $(GL(N),q)$-opers in a gauge where $\cL$ is $T$-invariant, the following conditions for any $N\geq2$ are satisfied 

\begin{align}\label{eq:Qksymmetry}
Q^+_k(z)=Q^+_k\left(\frac{1}{q^{k-1}z}\right),&&1\leq k\leq N,    
\end{align}
and
\begin{align}
\xi_N\left(\frac{1}{qz}\right)\xi_{N-1}\left(\frac{1}{qz}\right)=-\xi_i(z)\xi_i \left(\frac{1}{qz}\right), && 1\leq i \leq N,
\label{eq:GL2 embedding to GLN}
\end{align}
and, if $N\geq3$,
\begin{align}
&\xi_{N}\left(\frac{1}{q^{2}z}\right)\xi_{N-1}\left(\frac{1}{q^{2}z}\right)=\prod_{i=1}^{2}\xi_{N-2}\left(\frac{1}{q^{i}z}\right), &&
\\
&\xi_{N}\left(\frac{1}{q^{k-1}z}\right)\xi_{N-1}\left(\frac{1}{q^{k-1}z}\right)=\dfrac{\prod_{i=1}^{k-1}\xi_{N+1-k}\left(\dfrac{1}{q^{i}z}\right)}{\prod_{i=1}^{k-3}\xi_{N+2-k}\left(\dfrac{1}{q^{i+1}z}\right)},&& 4\leq k\leq N.
\label{eq:gl_n_const}
\end{align} 

\end{Prop}
\begin{proof}  
For $k=1$, by the definition of reflection-invariance and Remark \ref{determinant} we have
\begin{align}
    Q^{+}_{1}(z)=Q^{+}_{1}\left(\frac{1}{z}\right).
\end{align}
For $k=2$,  the reflection invariance provides a constraint on $Z(z)$ as follows. Using the argument as in Proposition \ref{Prop:ReflectionInv} we have 
\begin{align}
&\mathcal{D}_2\left(\frac{1}{q z}\right)=-\det(Z(1/qz))\xi^{-1}_{1}(1/qz)\cdots\xi^{-1}_{N-2}(1/qz)\nonumber
\\
&\times\left(e_{1}\wedge\cdots \wedge e_{N-2}\wedge s(qz)\wedge Z^{-1}(1/qz)s(z)\right).
\end{align}
By the parametric argument as in Proposition \ref{Prop:ReflectionInv} we have the following constraint on $Z(z)$ that is equivalent to $(3.18)$:
\begin{equation}\label{eq:sln}
    Z^{-1}(1/qz)=-\xi_N^{-1}(1/qz)\xi_{N-1}^{-1}(1/qz) Z(z),\qquad Q^+_2(z)=Q^+_2(1/qz),
\end{equation}
where $Z(z)=\diag(\xi_1(z),\cdots,\xi_N(z))$. 

For $3\leq k\leq N$, we consider the $q$-Wronskian at $1/z$:
\begin{align}
    \mathcal{D}_k\left(\frac{1}{z}\right)=\left(\bigwedge_{i=1}^{N-k}e_i\right)\wedge\left(\bigwedge_{j=0}^{k-1}\left(\prod_{m=0}^{j-1} Z\left(\frac{q^{k-2-m}}{z}\right)\right)\cdot s\left(\frac{q^{k-1-j}}{z}\right)\right).
\end{align}
By repeatedly using $s(1/z)=s(z)$ and replacing $z$ by $qz$ (in the same order), we arrive at
\begin{align}
&\mathcal{D}_k\left(\frac{1}{q^{k-1}z}\right)  =\left(\bigwedge_{i=1}^{N-k}e_i\right)\wedge\left(\bigwedge_{j=0}^{k-1}\left(\prod_{m=0}^{j-1} Z\left(\frac{1}{q^{m+1}z}\right)\right)\cdot s\left(q^{j}z\right)\right).
\end{align}
By Lemma \ref{Th:DetTransform}, 
\begin{align}
&\mathcal{D}_k\left(\frac{1}{q^{k-1}z}\right)  =    \left(\prod_{m=1}^{k-1}\det\left[ Z\left(\frac{1}{q^{m}z}\right)\right]\right)\bigwedge_{i=1}^{N-k}\left(\prod_{m=1}^{k-1} Z^{-1}\left(\frac{1}{q^{m}z}\right)\right)e_i
\\
&\wedge \left(\bigwedge_{j=0}^{k-1}\left(\prod_{m=0}^{k-j-2} Z^{-1}\left(\frac{1}{q^{k-1-m}z}\right)\right)\cdot s\left(q^{j}z\right)\right).\nonumber
\end{align}
By \eqref{eq:GL2 embedding to GLN}, we substitute $ Z^{-1}(1/q^{m}z)=-\xi_N^{-1}(1/q^{m}z)\xi_{N-1}^{-1}(1/q^{m}z) Z(q^{m-1}z)$ into the above equation and get
\begin{align}
 &\mathcal{D}_k\left(\frac{1}{q^{k-1}z}\right)=\left(\prod_{m=1}^{k-1}\det\left[ Z\left(\frac{1}{q^{m}z}\right)\right]\right)\left(\prod_{m=0}^{k-2}\prod_{i=1}^{N-k}\xi_i^{-1}\left(\frac{1}{q^{k-1-m}z}\right)\right)\left(\bigwedge_{i=1}^{N-k}e_i\right)
 \\
 &\wedge\left(\bigwedge_{j=0}^{k-1}\left(\prod_{m=0}^{k-j-2}  -\xi_N^{-1}\left(\frac{1}{q^{k-1-m}z}\right)\xi_{N-1}^{-1}\left(\frac{1}{q^{k-1-m}z}\right) Z(q^{k-2-m}z)\right)\cdot s\left(q^{j}z\right)\right), \nonumber
\end{align}
since $Z^{-1}(z)e_i=\xi^{-1}_{i}(z)e_i$.

Manipulation of the above formula gives 
\begin{align}
 &\mathcal{D}_k\left(\frac{1}{q^{k-1}z}\right)=\left(\prod_{j=0}^{k-2}\prod_{i=1}^{k}\xi_{N+1-i}\left(\frac{1}{q^{j+1}z}\right)\right)\prod_{j=0}^{k-2}\prod_{i=1}^{k-1-j} \left( -\xi_N^{-1}\left(\frac{1}{q^{j+i}z}\right)\xi_{N-1}^{-1}\left(\frac{1}{q^{j+i}z}\right) \right)
 \\
 &\times\left(\bigwedge_{i=1}^{N-k}e_i\right)\wedge \bigwedge_{j=0}^{k-1}\left(\left(\prod_{m=0}^{k-j-2}Z(q^{k-2-m}z)\right)\cdot s\left(q^{j}z\right)\right).\nonumber
\end{align}
Rearranging the order of the wedge product yields $(-1)^{k(k-1)/2}$, and we have
\begin{align}
 \mathcal{D}_k\left(\frac{1}{q^{k-1}z}\right)&=C_k(z)\mathcal{D}_k(z),
\end{align}
where
\begin{align}
C_k=(-1)^{\frac{k(k-1)}{2}}\prod_{j=0}^{k-2}\left[\left(\prod_{i=1}^{k}\xi_{N+1-i}\left(\frac{1}{q^{j+1}z}\right)\right)\left(\prod_{i=1}^{k-1-j}\left(-\xi^{-1}_{N}\left(\frac{1}{q^{j+i}z}\right)\xi^{-1}_{N-1}\left(\frac{1}{q^{j+i}z}\right)\right)\right)\right].
\label{eq:glnconst1}
\end{align} 
By the parametric argument as in Proposition \ref{Prop:ReflectionInv} we have
\begin{align}
    C_k(z)\mathcal{D}_k(z)=\mathcal{D}_k(z), && Q_{k}^{+}(z)=Q_{k}^{+}\left(\frac{1}{q^{k-1}z}\right),
\end{align}
therefore $C_k(z)=1$.

Since $C_k(z)$s are defined recursively, we have
\begin{align}
\mathcal{C}_k(z)
:=\frac{C_k(z)}{C_{k-1}(z)}
=
\left[\prod_{i=1}^{k-1}\xi_{N+1-i}\left(\frac{1}{q^{k-1}z}\right)\xi_{N+1-k}\left(\frac{1}{q^{i}z}\right)\right]\left[\xi^{-1}_{N}\left(\frac{1}{q^{k-1}z}\right)\xi^{-1}_{N-1}\left(\frac{1}{q^{k-1}z}\right)\right]^{k-1}
=1.
\end{align}
Taking the ratio $\cC_k(z)/\cC_{k-1}(qz)=1$, we obtain the following constraints
\begin{align}
&\xi_{N}\left(\frac{1}{q^{k-1}z}\right)\xi_{N-1}\left(\frac{1}{q^{k-1}z}\right)=\left\{\begin{array}{cc}
   \prod_{i=1}^{k-1}\xi_{N+1-k}\left(\dfrac{1}{q^{i}z}\right),  & k=3 \\
   & \\
    \dfrac{\prod_{i=1}^{k-1}\xi_{N+1-k}\left(\dfrac{1}{q^{i}z}\right)}{\prod_{i=1}^{k-3}\xi_{N+2-k}\left(\dfrac{1}{q^{i+1}z}\right)}, & 4\leq k\leq N
\end{array}\right..
\end{align}
\end{proof}
\begin{Cor}
For the boundary condition $Q_{N}^{+}(z)=\L(z)$, we have
\begin{align}
    \L(z)= \L\left(\frac{1}{q^{N-1}z}\right).
\end{align}
\end{Cor}
\qed

The last equation, in particular, implies that 
\begin{equation}
    \Lambda(z)=\gamma_N\prod_{j=1}^{L}\left(z-a_{j}\right)\left(\frac{1}{q^{N-1}z}-a_{j}\right)
\end{equation}
for some integer $L$.

\vskip.1in
It helps to look at some lower rank examples in order to gain some intuition.

\subsection{Example: \texorpdfstring{$(GL(3),q)$}{(GL(3),q)}-Oper}
We have the complete flag of subbundles as $\mathcal{L}_1\subset\mathcal{L}_2\subset E$ and Span $s(z)=\mathcal{L}_1$.
The two conditions for the Miura $(GL(3),q)$-oper read
\begin{align}\label{eq:MiuraOperSL3}
    e_1\wedge s(qz)\wedge A(z)s(z)&= Q^+_2(z)\,,\cr
    s(q^2z)\wedge A(qz) s(qz)\wedge A(qz) A(z)s(z) &= \Lambda(z)\,.
\end{align}
Consider them both at $1/z\in V\subset\mathbb{P}^1$
\begin{align}
e_1\wedge s\!\left(\frac{q}{z}\right)\wedge A\!\left(\frac{1}{z}\right)s\!\left(\frac{1}{z}\right) 
&= Q^+_2\!\left(\frac{1}{z}\right)\,,\\
s\!\left(\frac{q^2}{z}\right)\wedge A\!\left(\frac{q}{z}\right) s\!\left(\frac{q}{z}\right)\wedge 
A\!\left(\frac{q}{z}\right) A\!\left(\frac{1}{z}\right)s\!\left(\frac{1}{z}\right) 
&= \Lambda\!\left(\frac{1}{z}\right)\,.\nonumber
\end{align}
In the gauge where $s(z)$ is $T$-invariant we get
\begin{align}
e_1\wedge s\!\left(\frac{q}{z}\right)\wedge A\!\left(\frac{1}{z}\right)s\!\left(z\right) 
&= Q^+_2\!\left(\frac{1}{z}\right)\,,\\
s\!\left(\frac{q^2}{z}\right)\wedge A\!\left(\frac{q}{z}\right) s\!\left(\frac{q}{z}\right)\wedge 
A\!\left(\frac{q}{z}\right) A\!\left(\frac{1}{z}\right)s\!\left(z\right) 
&= \Lambda\!\left(\frac{1}{z}\right)\,.\nonumber
\end{align}
Using Lemma \ref{Th:DetTransform} we get 
\begin{align}
\text{det} A\!\left(\frac{1}{z}\right)\cdot A^{-1}\!\left(\frac{1}{z}\right)e_1\wedge A^{-1}\!\left(\frac{1}{z}\right)s\!\left(\frac{q}{z}\right)\wedge s\!\left(z\right) 
&= Q^+_2\!\left(\frac{1}{z}\right)\,,\\
\text{det} A\!\left(\frac{1}{z}\right)A\!\left(\frac{q}{z}\right)\cdot A^{-1}\!\left(\frac{q}{z}\right)A^{-1}\!\left(\frac{1}{z}\right)s\!\left(\frac{q^2}{z}\right)\wedge A^{-1}\!\left(\frac{1}{z}\right) s\!\left(\frac{q}{z}\right)\wedge 
s\!\left(z\right) 
&= \Lambda\!\left(\frac{1}{z}\right)\,.\nonumber
\end{align}
Shifting both equations by $q$ we get
\begin{align}
\text{det} A\!\left(\frac{1}{qz}\right)\cdot A^{-1}\!\left(\frac{1}{qz}\right)e_1\wedge A^{-1}\!\left(\frac{1}{qz}\right)s\!\left(\frac{1}{z}\right)\wedge s\!\left(qz\right) 
&= Q^+_2\!\left(\frac{1}{qz}\right)\,,\\
\text{det} A\!\left(\frac{1}{qz}\right)A\!\left(\frac{1}{z}\right)\cdot A^{-1}\!\left(\frac{1}{z}\right)A^{-1}\!\left(\frac{1}{qz}\right)s\!\left(\frac{q}{z}\right)\wedge A^{-1}\!\left(\frac{1}{qz}\right) s\!\left(\frac{1}{z}\right)\wedge 
s\!\left(qz\right) 
&= \Lambda\!\left(\frac{1}{qz}\right)\,.\notag
\end{align}
Using the $T$-invariance again we get
\begin{align}
\text{det} A\!\left(\frac{1}{qz}\right)\cdot A^{-1}\!\left(\frac{1}{qz}\right)e_1\wedge A^{-1}\!\left(\frac{1}{qz}\right)s\!\left(z\right)\wedge s\!\left(qz\right) 
&= Q^+_2\!\left(\frac{1}{qz}\right)\,,\\
\text{det} A\!\left(\frac{1}{qz}\right)A\!\left(\frac{1}{z}\right)\cdot A^{-1}\!\left(\frac{1}{z}\right)A^{-1}\!\left(\frac{1}{qz}\right)s\!\left(\frac{q}{z}\right)\wedge A^{-1}\!\left(\frac{1}{qz}\right) s\!\left(z\right)\wedge 
s\!\left(qz\right) 
&= \Lambda\!\left(\frac{1}{qz}\right)\,.\notag
\end{align}
Finally, we need to perform another $q$-shift in the second equation and use the $T$-invariance to get
\begin{equation}
\text{det} A\!\left(\frac{1}{q^2z}\right)A\!\left(\frac{1}{qz}\right)\cdot A^{-1}\!\left(\frac{1}{qz}\right)A^{-1}\!\left(\frac{1}{q^2z}\right)s\!\left(z\right)\wedge A^{-1}\!\left(\frac{1}{q^2z}\right) s\!\left(qz\right)\wedge 
s\!\left(q^2z\right) 
= \Lambda\!\left(\frac{1}{q^2z}\right)\,.\notag
\end{equation}
In other words, after reshuffling we get
\begin{align}
-\text{det} A\!\left(\frac{1}{qz}\right)\cdot A^{-1}\!\left(\frac{1}{qz}\right)e_1\wedge s\!\left(qz\right) \wedge A^{-1}\!\left(\frac{1}{qz}\right)s\!\left(z\right)
&= Q^+_2\!\left(\frac{1}{qz}\right)\,,\\
-\text{det} A\!\left(\frac{1}{q^2z}\right)A\!\left(\frac{1}{qz}\right)\cdot s\!\left(q^2z\right) \wedge A^{-1}\!\left(\frac{1}{q^2z}\right) s\!\left(qz\right)\wedge A^{-1}\!\left(\frac{1}{qz}\right)A^{-1}\!\left(\frac{1}{q^2z}\right)s\!\left(z\right)
&= \Lambda\!\left(\frac{1}{q^2z}\right)\,.\notag
\end{align}
The above equations must match the equations in \eqref{eq:MiuraOperSL3}. From the first equation we gain
\begin{equation}\label{formula for inverse}
     A^{-1}\!\left(\frac{1}{qz}\right) = -\alpha_1\!\left(\frac{1}{qz}\right)\text{det} A\!\left(\frac{1}{qz}\right)^{-1}\cdot A(z)\,,\quad Q^+_2\!\left(\frac{1}{qz}\right) = Q^+_2\!\left(z\right)\,.
\end{equation}
Now we substitute the expression for $A^{-1}\!\left(\frac{1}{qz}\right)$ into the second equation to get
\begin{align}
    -\text{det} A\!\left(\frac{1}{q^2z}\right)\text{det}A\!\left(\frac{1}{qz}\right) &\cdot s\!\left(q^2z\right) \nonumber\\
    \wedge \left(-\alpha_1\!\left(\frac{1}{q^2z}\right)\text{det} A\!\left(\frac{1}{q^2z}\right)^{-1}\cdot A(qz)\right) &\cdot s\!\left(qz\right)
    \\
    \wedge \alpha_1\!\left(\frac{1}{qz}\right)\text{det} A\!\left(\frac{1}{qz}\right)^{-1} \alpha_1\!\left(\frac{1}{q^2z}\right)\text{det} A\!\left(\frac{1}{q^2z}\right)^{-1}\cdot  A(z)A(qz)&\cdot s\!\left(z\right)
= \Lambda\!\left(\frac{1}{q^2z}\right)\nonumber
\end{align}

In the gauge where $A(z)=\text{diag}(\xi_1,\xi_2,\xi_3)(z)$ is diagonal, we have det$A(z)=\xi_1\xi_2\xi_3(z)$ for the generalized $Z$-twisted oper. After the cancellations in the above equations we get
\begin{equation}
    \xi_1\!\left(\frac{1}{qz}\right)\xi_1\!\left(\frac{1}{q^2z}\right)^2\text{det} A\!\left(\frac{1}{q^2z}\right)^{-1}=1\,,\quad
    \Lambda\!\left(\frac{1}{q^2z}\right)= \Lambda(z)\,.
\end{equation}
the first equation above is equivalent to
\begin{equation}
    \xi_1\!\left(\frac{1}{z}\right)\xi_1\!\left(\frac{1}{qz}\right) = \xi_2\!\left(\frac{1}{qz}\right)\xi_3\!\left(\frac{1}{qz}\right)
\end{equation}
Meanwhile, the first equation in \eqref{formula for inverse} yields 
\begin{equation}
    \xi_i(z)\xi_i\left(\frac{1}{qz}\right) = -\xi_2(z)\xi_3(z)\,,\qquad i=1,2,3\,.
\end{equation}
so the relations read 
\begin{equation}\label{eq:ReflCondGL3}
    \xi_{1}\left(\frac{1}{z}\right)=-\xi_1(z), \quad \xi_{2}\left(\frac{1}{qz}\right)=-\xi_{3}(z), \quad \xi_{2}(z)\xi_{3}(z)=\xi_{1}(z)\xi_{1}(qz).
\end{equation}
From the third equation, we have
\begin{equation}
\xi_{3}(z)\xi_{3}\left(\frac{1}{qz}\right)=\xi_{2}(z)\xi_{2}\left(\frac{1}{qz}\right)=\xi_{1}(z)\xi_{1}\left(\frac{1}{qz}\right).
\end{equation}

\subsubsection*{Example}
Let $f(z):=\frac{\xi_{1}}{\xi_{2}}(z).$ Then the above equation is equivalent to $f(\frac{1}{qz})=f(z)^{-1}$.
When $f(z)=1$, that is, when $\xi_{1}(z)=\xi_{2}(z)$, we can take 
\begin{equation}
\xi_{1}(z)=a\left(z-1/z\right)\prod_{j=1}^{m}\left(z-c_{j}\right)\left(1/z-c_{j}\right),
\end{equation}
where $a\in\mathbb{C}^\times$.

\subsection{Constraints for \texorpdfstring{$GL(4)$}{GL(4)}}
First, we have the analogue of \eqref{formula for inverse}:
\begin{equation}
    A^{-1}\!\left(\frac{1}{qz}\right) = -\xi_1\!\left(\frac{1}{qz}\right)\xi_2\!\left(\frac{1}{qz}\right)\text{det} A\!\left(\frac{1}{qz}\right)^{-1}\cdot A(z)\,.
\end{equation}
Next, we have 
\begin{equation}
    \xi_2\!\left(\frac{1}{z}\right)\xi_2\!\left(\frac{1}{qz}\right) = \xi_3\!\left(\frac{1}{qz}\right)\xi_4\!\left(\frac{1}{qz}\right).
\end{equation}
The third constraint reads
\begin{align}
    &\det \left(A\left(\frac{1}{qz}\right)A\left(\frac{1}{q^2z}\right)A\left(\frac{1}{q^3z}\right)\right)\cr
    &=
    \left(\xi_3\!\left(\frac{1}{q^3z}\right)\xi_4\!\left(\frac{1}{q^3z}\right)\right)^{3}
    \left(\xi_3\!\left(\frac{1}{q^2z}\right)\xi_4\!\left(\frac{1}{q^2z}\right)\right)^{2}
    \left(\xi_3\!\left(\frac{1}{qz}\right)\xi_4\!\left(\frac{1}{qz}\right)\right).
\end{align}

The above constraints are reproduced from Proposition \ref{Th:PropInvA} by taking $N=4$.
\begin{align}
   - \xi_{i}(z)\xi_{i}\left(\frac{1}{qz}\right)=\xi_{3}\left(\frac{1}{qz}\right) \xi_{4}\left(\frac{1}{qz}\right),&& 1\leq i\leq 4\,.
\end{align}
\begin{align}
   -\xi_{2}\left(z\right)=\xi_{2}\left(\frac{1}{z}\right),
&&
\xi_{2}(z)=\frac{\xi_{1}\left(z\right)\xi_{1}\left(qz\right)}{\xi_{1}\left(\frac{1}{z}\right)}     \,.
\end{align}
These relations lead to an equation for $\xi_1(z)$,
\begin{align}
\xi_{1}\left(\frac{1}{z}\right)^2 \xi_{1}\left(\frac{q}{z}\right) =-\xi_{1}^2\left(z\right)\xi_{1}\left(qz\right)\,.
\end{align}

\section{$GL(N)$ Reflection Invariant Equations and Bethe Ansatz}\label{Sec:SolutionsReflEqns}
In this section, we study the space of solutions of the reflection invariant equations \eqref{eq:GL2 embedding to GLN} for the $q$-connections which admit at most simple pole at $0$ or $\infty$ similar to the $GL(2)$ case in Section \ref{GL2}. We can refer to such connection as \textit{tame}.

\subsection{Generalized Twists and Bethe Equations}
First, we would like to introduce the following parameterization for the eigenvalues of the $q$-connection $Z$,
\begin{equation}\label{eq:GeneralizedTw}
    \xi_k(z)=  \prod_{i=1}^{k-1} x_{N-i} (q^{k-N+1}z)\cdot x_{N}(q^k z)\,,
\end{equation}
where $x_i(z)$ are some new rational functions. This will allow us to express the generalized open Bethe equations \eqref{eq:bethe} in terms of these functions and to compare our results with the literature on integrable systems. Thus we get for \eqref{eq:bethe} the following explicit presentation of the generalized open Bethe equations,
\begin{align}\label{eq:BetheEqnsGLN}
    x_k(s_{k,i}) \cdot
    q^{p_{k}-p_{k-1}}
    \cdot&
    \prod_{j=1}^{p_{k-1}}\frac{qs_{k,i}-s_{k-1,j}}{s_{k,i}-s_{k-1,j}}\frac{q^{k-1}s_{k,i}s_{k-1,j}-1}{q^{k-2}s_{k,i}s_{k-1,j}-1}    
    \cr
    \cdot&\prod_{j=1}^{p_{k}}\frac{s_{k,i}-qs_{k,j}}{qs_{k,i}-s_{k,j}
    }\frac{q^{k-2}s_{k,i}s_{k,j}-1}{q^{k}s_{k,i}s_{k,j}-1}\cr
    \cdot&
\prod_{j=1}^{p_{k+1}}\frac{s_{k,i}-s_{k+1,j}}{s_{k,i}-qs_{k+1,j}}
\frac{q^{k}s_{k,i}s_{k+1,j}-1}{q^{k-1}s_{k,i}s_{k+1,j}-1}
=-1\,,    
\end{align}
for $k=1,\dots,N-1$ where $p_0=0$, and $s_{N,i}=a_i$.

Our results thus far can be formulated in the following theorem.

\begin{Thm}
Let $\xi_1(z),\dots,\xi_N(z)$ and correspondingly $x_1(z),\dots,x_{N-1}(z)$ satisfy the reflection invariant conditions \eqref{eq:GL2 embedding to GLN}.
Then there is a one-to-one correspondence between the set of nondegenerate generalized $Z$-twisted reflection-invariant Miura $(GL(N),q)$-opers and the set of nondegenerate solutions of the generalized open XXZ Bethe Ansatz equations \eqref{eq:BetheEqnsGLN}.
\end{Thm}

In other words, the space of generalized $Z$-twisted Miura $(GL(N),q)$-opers is described by \eqref{eq:BetheEqnsGLN} for arbitrary rational functions $x_1(z),\dots,x_{N-1}(z)$, whereas in order to describe the space of generalized $Z$-twisted \textit{reflection-invariant} Miura $(GL(N),q)$-opers we need to solve the reflection-invariant conditions \eqref{eq:GL2 embedding to GLN}.

The reflection invariant condition is an important building block in the study of Bethe Ansatz for open spin chains. The shape of generalized twists $x_i(s)$ in \eqref{eq:BetheEqnsGLN} is determined by the K-matrix which together with an R-matrix satisfies the reflection equation.

\begin{proof}
    From \eqref{eq:Qksymmetry} we get $Q_{k}^{+}(z)=\gamma_k
\prod_{j=1}^{p_{k}}\left(z-s_{k,j}\right)\left(\dfrac{1}{q^{k-1}z}-s_{k,j}\right)$, $1\leq k\leq N$ for some constants $\gamma_k$. Using \eqref{eq:GeneralizedTw} we can directly obtain \eqref{eq:BetheEqnsGLN} provided that the nondegeneracy conditions on the roots of all Laurent polynomials hold.
\end{proof}

\subsection{Solutions of the Reflection Invariant Conditions} 
Let us now discuss the solutions of the reflection invariant conditions in terms of functions $x_i(z)$. First, let us consider the low rank examples.

\subsubsection{$(GL(2),q)$-opers} 
For $N=2$ we get
\begin{equation}
   \xi_1(z)=x_2(q z)\,,\qquad \xi_2(z)=x_1(q z) x_2(q^2 z) \,,
\end{equation}
so $\dfrac{\xi_2(q^{-1}z)}{\xi_1(z)}=x_1(z)$.

The reflection invariance conditions \eqref{eq:AreflectionCompts} then read
\begin{equation}
    x_2(q z)=-x_1\left(\frac{1}{z}\right) x_2\left(\frac{q}{z}\right)\,,\qquad x_1(q z) x_2(q^2 z) = -x_2\left(\frac{1}{z}\right)\,,
\end{equation}
which can be thought of as a definition of $x_1$ in terms of $x_2$. Therefore, we get
\begin{equation}\label{eq:x1GL2}
    x_1(z)=-\frac{x_2\left(\dfrac{q}{z}\right)}{x_2(qz)}\,.
\end{equation}

\subsubsection{$(GL(3),q)$-opers} 
For $N=3$ we get
\begin{equation}\label{eq:GL3twists}
   \xi_1(z)=x_3(q z)\,,\quad \xi_2(z)=x_2(z) x_3(q^2 z)\,,\quad \xi_3(z)=x_1(q z) x_2(q z) x_3(q^3 z)\,,
\end{equation}
so $\dfrac{\xi_3(q^{-1}z)}{\xi_2(z)}=x_1(z)$ and $\dfrac{\xi_2(z)}{\xi_1(qz)}=x_2(z)$.
Then the reflection invariant conditions \eqref{eq:ReflCondGL3} after straightforward substitution of \eqref{eq:GL3twists} are equivalent to the following
\begin{equation}
    x_{3}\left(\frac{q}{z}\right)=-x_3(qz), \quad x_1(z)=-\frac{x_2\left(\dfrac{1}{z}\right)}{x_2(z)}\frac{x_3\left(\dfrac{q^2}{z}\right)}{x_3(q^2z)}, \quad x_2\left(\frac{1}{qz}\right)=x_2(z)^{-1}\,.
\end{equation}
We can rewrite them as
\begin{equation}\label{eq:GL3tele}
    x_1(z)=\frac{x_2\left(\frac{1}{z}\right)}{x_2(z)}\frac{x_3(z)}{x_3(q^2z)}\,, \quad  x_2\left(\frac{1}{qz}\right)=x_2(z)^{-1}\,,\quad x_{3}\left(\frac{q}{z}\right)=-x_3(qz)\,.
\end{equation}

\subsubsection{General Case} 
\eqref{eq:GL2 embedding to GLN} and \eqref{eq:gl_n_const} are rewritten as
\begin{align}\label{eq:tele_GL2 embedding to GLN}
&(1\leq i\leq N),\qquad- x_{N}\left(\frac{q^{N-1}}{z}\right) x_{N}\left(\frac{q^{N-2}}{z}\right)  x_{1}\left(\frac{1}{z}\right)  \prod_{j=1}^{N-2} x_{N-j}\left(\frac{1}{z}\right)  x_{N-j}\left(\frac{1}{qz}\right) 
\\
&= x_{N}(q^i z) x_{N}\left(\frac{q^{i-1}}{z}\right)\prod_{j=1}^{i-1} x_{N-j} (q^{i-N+1}z)x_{N-j}\left(\frac{q^{i-N}}{z}\right)\,,\qquad \nonumber
\end{align}
and 
\begin{align}
\prod_{i=1}^{k-3}x_{k-1}\left(\frac{1}{q^{i} z}\right)=   \prod_{i=1}^{k-3}x_{k-1}^{-1}\left(\frac{z}{q^{i} }\right)\,,
&&\frac{x_{N}\left(\dfrac{q^{N+1-k} }{z}\right)}{x_{N}(q^{N+1-k}z)}= - \prod_{i=1}^{k-3}x_{k-1}\left(\frac{1}{q^{i} z}\right)\prod_{i=1}^{N-k}\frac{ x_{N-i}(q^{2-k}z) }{ x_{N-i}\left(\dfrac{q^{2-k} }{z}\right) }  \,,
\end{align}
where $3\leq k \leq N$. These provide several expressions of $x_1(z)$. Some constraints come from this redundancy. 
\begin{align}
x_{1} (z)=    -\frac{x_{N}\left(\dfrac{q^{N-1}}{z}\right)    \prod_{i=1}^{N-2}  x_{N-i}\left(\dfrac{1}{z}\right)}{x_{N}(q^{N-1}z) \prod_{i=1}^{N-2} x_{N-i} (z)}\,. 
\end{align}

\subsection{Imposing Tame Singularity Conditions}
Finally, we impose the condition that the $q$-oper connection $A(z)$ exhibits at most simple pole behavior (i.e. either a simple pole or a constant) at $0$ and $\infty$ on $\mathbb{P}^1$. We list the results for $GL(2)$ and $GL(3)$ opers below.

\subsubsection{$(GL(2),q)$-opers} 
Assume $A(z)$ has constant asymptotics at zero and infinity. We add two poles for $\xi_1(z)$ at the fixed points of the involution $z\mapsto \frac{1}{qz}$ which result in two zeros 
\begin{equation}
    \xi_1(z)=c\frac{(z-d_+)(z-d_-)}{qz^2-1}
\end{equation}
so
\begin{equation}\label{eq:tameGL2}
    x_2(z)=\frac{c}{q}\frac{(z-qd_+)(z-qd_-)}{z^2-q}\,,\qquad x_1(z)=\frac{(d_- z-1) (d_+ z-1) \left(q
   z^2-1\right)}{(d_--z) (z-d_+) \left(z^2-q\right)}
\end{equation}

\subsubsection{$(GL(3),q)$-opers: General Solution} 
Equation $x_{3}(\frac{q}{z})=-x_3(qz)$ can be rewritten via $h(z)=x_3(qz)$ as $h(z)=-h(\frac{1}{z})$. Its rational solution can be represented as
\begin{equation}
h(z)=\left(\frac{z^2-1}{z}\right)F\left(z+\frac{1}{z}\right),
\end{equation}
where $F(z)=P(z)/Q(z)$ is any rational function.
Explicitly, $h(z)$ is of the form
\begin{equation} 
h(z)=C\left(\frac{z^2-1}{z}\right)
\; z^{\sum_j n_j - \sum_i m_i}
\;
\frac{
\prod_i (z-a_i)^{m_i}(z-1/a_i)^{m_i}
}{
\prod_j (z-b_j)^{n_j}(z-1/b_j)^{n_j},
}
\end{equation}
where $C$ is any constant.
Hence,
\begin{equation}
x_{3}(z)=
C\,
\left(\frac{z^2 - q^2}{z}\right)
z^{\sum_{j} n_{j} - \sum_{i} m_{i}}
\;
\frac{
\prod_{i} (z-a_i q)^{m_i}(z-q/a_i)^{m_i}
}{
\prod_{j} (z-b_j q)^{n_j}(z-q/b_j)^{n_j}
}.
\end{equation}
For $x_{2}(z)$ which satisfies $x_{2}(z)x_{2}(1/(qz))=1$, we have
\begin{equation}
x_{2}(z)=
\pm
q^{\frac{k}{2}+\sum k_{i}}z^{k+\sum k_{i}}
\prod_{i}
\left(
\frac{z - c_i}{1 - c_i q z}
\right)^{k_i},
\end{equation}
where $k$ is any integer. Note that when $k$ and $\sum k_{i}$ are of the same parity, $x_{2}(z)$ comes from the Ansatz $x_{2}(z)=\pm\frac{P(z)}{P(1/(qz))}$, where $P(z)$ is a rational function.
As $x_{1}(z)$ is given by 
\[
x_1(z)=\frac{x_2\left(\frac{1}{z}\right)}{x_2(z)}\frac{x_3(z)}{x_3(q^2z)}\,,
\]
an explicit formula for $x_{1}(z)$ is given by
\begin{equation}
\begin{aligned}
x_1(z)
&= z^{-2k-2\sum_i k_i}
\prod_i
\left[
\frac{(1-c_iz)(1-qc_iz)}{(z-qc_i)(z-c_i)}
\right]^{k_i}
\\[6pt]
&\quad \times
\frac{z^2-q^2}{q^2z^2-1}
\cdot q^{-2\left(\sum_j n_j-\sum_i m_i\right)}
\\[6pt]
&\quad \times
\prod_i
\left(
\frac{(z-a_i q)(z-q/a_i)}
{(q^2 z-a_i q)(q^2 z-q/a_i)}
\right)^{m_i}
\\[6pt]
&\quad \times
\prod_j
\left(
\frac{(q^2 z-b_j q)(q^2 z-q/b_j)}
{(z-b_j q)(z-q/b_j)}
\right)^{n_j}.
\end{aligned}
\end{equation}

\subsubsection{Examples of Solutions}\label{Sec:List}
We shall list several examples of the above solutions below
\begin{enumerate}[(a)]
    \item 
The minimal solution is $x_2(z)=1$ and 
\begin{equation}
    x_3(z)=c\frac{z^2-q^2}{qz}\,.
\end{equation}
In this case
\begin{equation}
    x_1(z)=\frac{x_3(z)}{x_3(q^2 z)}
\end{equation}
and only consists of the contact term. 
\vskip.1in
For the following solutions, we begin with specifying the components $H(z)$ and $F(z)$ for $x_{3}(z)$ and polynomial $P(z)$ for $x_{2}(z)$ written in the form
\begin{equation}
    x_3(z)=F\left(z+\frac{q^2}{z}\right)\left(H(z)-H\left(\frac{q^2}{z}\right)\right)\,,\qquad x_2(z)=\frac{P(z)}{P(\frac{1}{qz})}\,.
\end{equation}
All of the following solutions provide the tame $q$-connection. 
\item $H(z)=\dfrac{z-b}{z-a}+cz$, $P(z)=\dfrac{z-e}{z-d}$,  $F(z)=1$. 
\begin{align}
&x_3(z)=\frac{q^2-b z}{a z-q^2}+\frac{b-z}{a-z}-\frac{c q^2}{z}+c z,
\qquad
x_2(z)=\frac{(e-z) (d q z-1)}{(d-z) (e q z-1)}, 
\\
&x_1(z)=\frac{(a z-1) (z-d) (e z-1) \left(z^2-q^2\right) \left(a-q^2 z\right) (z-d q) (e q z-1) }{(a-z) (d z-1) (z-e) \left(q^2 z^2-1\right) \left(a z-q^2\right) (d q z-1) (z-e q)}
\nonumber
\\
&\times \frac{ \left(a^2 c z-a \left(c \left(q^2+z^2\right)+z\right)+z \left(b+c q^2\right)\right)}{\left(a^2 c z-a \left(c (q^2 z^2+1)+z\right)+z \left(b+c q^2\right)\right)}.
\nonumber
\end{align}

\item $H(z)=\dfrac{z-b}{z-a}+cz$, $P(z)=dz+\dfrac{e}{z}+f$,  $F(z)=1$. 
\begin{align}
&x_3(z)=\frac{q^2-b z}{a z-q^2}+\frac{b-z}{a-z}-\frac{c q^2}{z}+c z,
\qquad
x_2(z)=\frac{q (z (d z+f)+e)}{d+q z (e q z+f)}, 
\\
&x_1(z)=\frac{(a z-1) \left(z^2-q^2\right) \left(a-q^2 z\right) (d+z (e z+f)) }{(a-z) \left(q^2 z^2-1\right) \left(a z-q^2\right) (z (d z+f)+e) }
\nonumber
\\
&\times\frac{ \left(a^2 c z-a \left(c \left(q^2+z^2\right)+z\right)+z \left(b+c q^2\right)\right) (d+q z (e q z+f))}{ \left(a^2 c z-a \left(c( q^2 z^2+1)+z\right)+z \left(b+c q^2\right)\right) \left(z (d z+f q)+e q^2\right)}.
\nonumber
\end{align}

\item $H(z)=A\,z^2$,\, $F(z)=\left[\left(\frac{q^2}{z}+z\right)(a-z) (c-z)  \left(a-\frac{q^2}{z}\right) \left(c-\frac{q^2}{z}\right)\right]^{-1}$,

$P(z)=\dfrac{z}{q^2}-\dfrac{1}{z}$.
\begin{align}\label{eq:typeD}
&x_3(z)=\frac{A\, z (q-z) (q+z)}{(a-z) (z-c) \left(a z-q^2\right) \left(c z-q^2\right)},\qquad x_2(z)=\frac{q \left(z^2-q^2\right)}{q^4 z^2-1},
\\
&x_1(z)=\frac{\left(q^4 z^2-1\right)}{\left(z^2-q^4\right) }\frac{(a z-1) (c z-1)  \left(a-q^2 z\right) \left(q^2 z-c\right)}{(a-z) (z-c) \left(a z-q^2\right) \left(c z-q^2\right)}.
\nonumber
\end{align}

\end{enumerate}

\subsection{Imposing Attenuation Constraints} 
Finally we would like to explore the asymptotic behavior of our twist solutions $\xi_i(z)$ at zero and infinity.

\subsubsection{$(GL(2),q)$-Opers}
From \eqref{eq:tameGL2} we get 
\begin{equation}
    \xi_1(z)=\frac{c (d_--z) (z-d_+)}{q z^2-1}\,,\qquad \xi_2(z)=-\frac{c
   (d_- q z-1) (d_+ q z-1)}{q \left(q
   z^2-1\right)}
\end{equation}
Asymptotically we get for $Z(z)=\text{diag}(\xi_1(z),\xi_2(z))$
\begin{equation}
    Z(0)=c\begin{pmatrix}
        - d_+d_- &0\\
        0& q^{-1}
    \end{pmatrix}
    \,,\qquad
     Z(\infty)=-c\begin{pmatrix}
       q^{-1}  &0\\
        0& - d_+d_-
    \end{pmatrix}
\end{equation}
which can be written as
\begin{equation}
    Z(0)=c\,\upxi_0^{1/2}\cdot \upxi_1^{\frac{\check{\alpha}}{2}},\qquad Z(\infty)=c\,\upxi_0^{1/2}\cdot \upxi_1^{-\frac{\check{\alpha}}{2}}
\end{equation}
where $\upxi_0=-\dfrac{d_+d_-}{q}$ and $\upxi_1=-d_+d_-q$.

\subsubsection{$(GL(3),q)$-Opers}
The following itemized list follows the same numbers as in Subsubsection \ref{Sec:List}

\begin{enumerate}[(a)]
\item 
    \begin{align}
    \xi_1(z)=\xi_3(z)=c\left(z-\frac{1}{z}\right),&& \xi_2(z)=c\left(q z-\frac{1}{q z}\right)\,. 
    \end{align}

  \begin{align}
Z(z)\sim -c\left(   \begin{array}{ccc}
          1   &0 &0 \\
          0   & q^{-1} &0 \\
            0 &  0& 1  
        \end{array}\right)\frac{1}{z}, &&
  Z(z)\sim c\left(   \begin{array}{ccc}
          1     &0 &0 \\
          0   & q    &0 \\
            0 &  0&1
        \end{array}\right)z
\end{align}
as $z$ approaches $0$ and $\infty$ respectively. We can see that the residue of the connection matrix approaches the above elements of $GL(3)\subset GL(3)(z)$.

\item 

\begin{align}
&\xi_1(z)=q \left(z^2-1\right) \left(\frac{b-a}{(q-a z) (q z-a)}+\frac{c}{z}\right)\,,\nonumber
\\
&\xi_2(z)=\frac{(e-z) \left(q^2 z^2-1\right) (d q z-1) \left(-z \left(a (a c-1)+b+c q^2\right)+a c q^2 z^2+a c\right)}{z (a z-1) (d-z) \left(q^2 z-a\right) (e q z-1)}\,,
\\
&\xi_3(z)=\frac{q \left(z^2-1\right) (d-z) (e q z-1) \left(a^2 c z-a \left(c q \left(z^2+1\right)+z\right)+z \left(b+c q^2\right)\right)}{z (z-e) (a z-q) (q z-a) (d q z-1)}\,.\nonumber
\end{align}

\begin{equation}
    Z(z)\sim\frac{Z^{(0)}_{-1}}{z}+Z^{(0)}_0+\dots\,,\qquad 
Z(z)\sim Z^{(\infty)}_{1}z+Z^{(\infty)}_0+\dots
\end{equation}
as  $z\to 0$ and $z\to\infty$ respectively, and where
\begin{align}
Z^{(0)}_{-1}=-cq\left(\begin{array}{ccc}
     1&  0&0\\
     0& q^{-1}\frac{e}{d}&0\\
     0&   0&  \frac{d}{e}\\
\end{array}\right),&& Z^{(\infty)}_{1}=cq\left(\begin{array}{ccc}
   1  &  0&0\\
     0&   q\frac{d}{e}&0\\
     0&   0&\frac{e}{d}\\
\end{array}\right).
\end{align}
Note that the two terms have the same $q$ factors with $d$ and $e$ interchanged.

\item
\begin{align}
&\xi_1(z)=q \left(z^2-1\right) \left(\frac{b-a}{(q-a z) (q z-a)}+\frac{c}{z}\right)\,,\nonumber
\\
&\xi_2(z)=\left(\frac{q z \left(d z+\frac{e}{z}+f\right)}{d+q z (e q z+f)}\right)\left(\frac{b-q^2 z}{a-q^2 z}+\frac{1-b z}{a z-1}+c q^2 z-\frac{c}{z}\right)\,,
\\
&\xi_3(z)=\frac{\left(z^2-1\right) \left(-z \left(a (a c-1)+b+c q^2\right)+a c q z^2+a c q\right) (d+q z (e q z+f))}{z (a z-q) (q z-a) (z (d z+f)+e)}\,.\nonumber
\end{align}

\begin{equation}
    Z(z)\sim\frac{Z^{(0)}_{-1}}{z}+Z^{(0)}_0+\dots\,,\qquad 
Z(z)\sim Z^{(\infty)}_{1}z+Z^{(\infty)}_0+\dots
\end{equation}
as $z\to 0$ and $z\to\infty$ respectively, where 
\begin{align}
Z^{(0)}_{-1}=-cq\left(\begin{array}{ccc}
     1&  0&0\\
     0& \frac{e}{d}&0\\
     0&   0&  q^{-1}\frac{d}{e}\\
\end{array}\right),&& Z^{(\infty)}_{1}=cq\left(\begin{array}{ccc}
   1  &  0&0\\
     0&   \frac{d}{e}&0\\
     0&   0&q\frac{e}{d}\\
\end{array}\right).
\end{align}

\item
\begin{align}
&\xi_1(z)  =\frac{A q z \left(z^2-1\right)}{(q-a z) (q z-a) (q-c z) (q z-c)},
\\
&\xi_2(z)=\frac{A q z (q-z) (q+z) \left(q^2 z^2-1\right)}{(a z-1) (c z-1) \left(q^4 z^2-1\right) \left(a-q^2 z\right) \left(q^2 z-c\right)},
\\
&\xi_3(Z)=-\frac{A z \left(z^2-1\right) \left(q^4 z^2-1\right)}{\left(q^2-z^2\right) (q-a z) (q z-a) (q-c z) (q z-c)}\,.
\end{align}

\begin{equation}
    Z(z)\sim Z_1^{(0)}z+\dots\,,\qquad 
Z(z)\sim Z_{-1}^{(\infty)}z^{-1}+\dots,
\end{equation}
as $z\to 0$ and $z\to\infty$ respectively and where
\begin{align}
Z_1^{(0)}= -\frac{A}{a c q}+
\left(\begin{array}{ccc}
          1   &0 &0 \\
          0   & q^4 &0 \\
            0 &  0&q^{-3} 
        \end{array}\right),
&&
Z_{-1}^{(\infty)}=\frac{A}{a c q }\left(\begin{array}{ccc}
          1   &0 &0 \\
          0   & q^{-4} &0 \\
            0 &  0&q^{3} 
        \end{array}\right).
\end{align}

\end{enumerate}

\section{\texorpdfstring{$(GL(2),\epsilon)$}{(GL(2),epsilon)}-Opers and Open XXX spin chains}\label{epsilon opers}
In \cite{Koroteev:2022aa} the first author and A. Zeitlin studied families of difference and differential opers on $\mathbb{P}^1$. The $q$-opers are on top of the hierarchy that also contains $\epsilon$-opers as well as trigonometrically and rationally twisted differential opers. We plan to investigate the orbifolding construction of the entire family in the upcoming work, however, in this section, we present the story of orbifolded $(GL(2),\epsilon)$-opers.

\vskip.1in

Consider automorphism $M_\epsilon: \mathbb{P}^1\to\mathbb{P}^1$ acting as $z\mapsto  z+\epsilon$ for $\epsilon\in\mathbb{C}$ together with $\mathbb{Z}_2$ reflection $T'$ acting as $z\mapsto -z$. 

\begin{Def}
For a Zariski open dense subset $U\subset\mathbb{P}^1$ consider $V=U\cap M_\epsilon^{-1}(U)$. A meromorphic orbifolded $(GL(2),\epsilon)$-oper on $\mathbb{P}^1$ is a triple $(E,A,\mathcal{L})$, where $E$ is a rank-two holomorphic vector bundle on $\mathbb{P}^1$ such that it is isomorphic to its pullback $E\simeq E^{\epsilon}$ with respect to $M_\epsilon$, $\mathcal{L}\subset E$ is a line subbundle, and
$A\in\text{Hom}_{\mathcal{O}_{\mathbb{P}^1}}(E,E^\epsilon)$, such that the restriction 
\begin{equation}
\label{eq:orbopdefEps}
\bar{A}: \mathcal{L}\to \left(E/\mathcal{L}\right)^{\epsilon}
\end{equation}
is an isomorphism on $V$. 
\end{Def}

Similarly we define Miura $(GL(2),\epsilon)$-opers as a quadruple $(E,A,\mathcal{L},\hat{\mathcal{L}})$ where $\hat{\mathcal{L}}$ is preserved by $A(z)$.

\begin{Def}
The Miura $(GL(2),\epsilon)$-oper is called \textit{generalized $Z$-twisted} if there is a $\epsilon$-gauge transformation $v(z)\in B_+(z)$ such that
\begin{equation}\label{eq:MiuraGenEps}
A(z)=v(z+\epsilon)Z(z)v(z)^{-1}
\end{equation}
where $Z(z)\in H(z)$. In local coordinates,
\begin{equation}
\label{eq:qConnDiagep}
Z(z) = 
\begin{pmatrix}
\phi_1(z) & 0\\
0 & \phi_2(z) 
\end{pmatrix}\,,
\end{equation}
When $Z(z)\in H\subset H(z)$ we recover the $Z$-twisted $\epsilon$-oper.
\end{Def}

\begin{Def}\label{Def:REflInf2Eps}
    A generalized $Z$-twisted Miura $(GL(2),\epsilon)$-oper is called \textit{reflection-invariant} if there exists a $\epsilon$-gauge transformation $v(z)\in B_+(z)$ such that $s(z)=v(z)\begin{pmatrix}
        0\\
        1
    \end{pmatrix}$ is invariant under $t$: $s(-z)=s(z)$ and the $\epsilon$-connection takes the form \eqref{eq:qConnDiagep}.
\end{Def}

Let $s(z)=\begin{pmatrix}
    Q_-(z)\\Q_+(z)
\end{pmatrix}$ then the reflection-invariant condition will entail similarly to \eqref{eq:QSymDef} and \eqref{eq:LambdaRefl}
\begin{equation}
Q_+(-z)=Q_+(z)\,,\qquad \Lambda(-z-\epsilon)=\Lambda(z)\,,
\end{equation}
or, explicilty,
\begin{equation}
Q_{+}(z)=\prod_{k=1}^{m}\left(z-s_{k}\right)\left(-z-s_{k}\right)\,,\qquad \Lambda(z)=\prod_{k=1}^{L}\left(z-a_{k}\right)\left(-z-\epsilon-a_{k}\right)\,.
\end{equation}
For the $\epsilon$-connection the reflection-invariance implies that
\begin{equation}
    A\left(-z-\epsilon\right)A(z)=-\det A\left(-z-\epsilon\right)\cdot 1.
\end{equation}
This can be solved in diagonal components as
\begin{align}
    \phi_{1}(z)&=(2z+2\epsilon+1)\left(z+\epsilon+\mathrm{m}-\frac{1}{2}\right)\left(z+\epsilon+\tilde{\mathrm{m}}-\frac{1}{2}\right)\cr
    \phi_2(z)&=(2z-1)\left(z-\mathrm{m}+\frac{1}{2}\right)\left(z-\tilde{\mathrm{m}}+\frac{1}{2}\right)
\end{align}

Analogously to the $q$-difference case, one proves \cite{Koroteev:2022aa}
\begin{Thm}
Then there is a one-to-one correspondence between the set of nondegenerate generalized $Z$-twisted reflection-invariant Miura $(GL(2),\epsilon)$-opers and the set of nondegenerate solutions of the following generalized $\mathfrak{gl}_2$ XXX Bethe equations
\begin{equation}\label{bethe from VWeps}
-\frac{\phi_1(s_{i})}{\phi_2(s_{i}-\epsilon)}\frac{Q_{+}(s_{i}+\epsilon)}{Q_{+}(s_{i}-\epsilon)}\frac{\Lambda(s_{i}-\epsilon)}{\Lambda(s_{i})}=1\,, \quad i=1,\dots,m
\end{equation}
\end{Thm}

In this construction, the Bethe equations take the form
\begin{align}\label{bethe using additive opers}
&\frac{\left(2 s_i+2 \epsilon +1\right)}{\left(2 s_i-2 \epsilon -1\right)}\frac{ \left(s_i+\mathrm{m}+\epsilon
   -\frac{1}{2}\right) \left(s_i+\tilde{\mathrm{m}}+\epsilon
   -\frac{1}{2}\right)}{
   \left(s_i+\mathrm{m}-\epsilon -\frac{1}{2}\right) \left(s_i+\tilde{\mathrm{m}}-\epsilon
   -\frac{1}{2}\right)}\cr
    &\cdot\prod_{k=1}^{m}\frac{(s_{i}-s_{k}+\epsilon)(s_{i}+s_{k}+\epsilon)}{(s_{i}-s_{k}-\epsilon)(s_{i}+s_{k}-\epsilon)}
    \prod_{j=1}^{L}
    \frac{(s_{i}+a_{j}-\epsilon)(-s_{i}+a_{j})}{(s_{i}+a_{j})(-s_{i}-\epsilon+a_{j})}=1, \quad 1\leq i\leq m.
\end{align}

\begin{Thm}
Then there is a one-to-one correspondence between the set of nondegenerate generalized $Z$-twisted reflection-invariant Miura $(GL(2),\epsilon)$-opers and the set of nondegenerate solutions of Bethe Ansatz equations \eqref{bethe using additive opers}.
\end{Thm}

A generalization to the system of equations describing $(GL(N),\epsilon)$-opers goes along the line of the previous section and will be reported in the upcoming work.

\section{Literature on Open Spin Chains}\label{Sec:Connect}
In this section, we show that the Bethe equations of XXZ (and XXX) open spin chains in the papers known to the authors fit into our construction based on the reflection-invariant generalized $Z$-twisted Miura $(GL(N),q)$-oper ($(GL(N),\epsilon)$-oper) conditions. 

Note that, in contrast to the extensive literature on spin chains with periodic boundary conditions, studies of open spin chains remain relatively limited both in physics and in representation theory\footnote{We encourage the reader to inform us about the other existing results in the field.}.

\subsection{XXZ \texorpdfstring{$U_q(\hat{\mathfrak{sl}}_2)$}{U_q(sl_2)} Open Chain by Sklyanin \cite{sklyaninboundary} and Vlaar-Weston \cite{Vlaar:2020jww}} 
The Bethe equations of XXZ $U_q(\hat{\mathfrak{sl}}_2)$ Open Chain with diagonal boundary conditions are derived in \cite{sklyaninboundary} via Algebraic Bethe Ansatz. Later, the Bethe equations are reproduced in \cite{Vlaar:2020jww} via representation theory of $U_q(\hat{\mathfrak{sl}}_2)$ and its Borel subalgebra as well as infinite dimensional solutions of the reflection equation.

In the following proposition, it is shown that the Bethe equations in \cites{sklyaninboundary,Vlaar:2020jww} fit into the oper construction. 

\begin{Prop}
The generalized XXZ Bethe equations given by \eqref{bethe from VW} reduce to the open Bethe Ansatz equations given in \cite[Theorem 6.8]{Vlaar:2020jww}:
\begin{equation}\label{equations from VW 1/z}
\frac{(1-\tilde{\xi}y_{i}^{2})}{(\tilde{\xi}-q^{2}y_{i}^2)}
\frac{(1-\xi y_{i}^{2})}{(\xi-q^{2}y_{i}^2)}  
\frac{(1-q^{2})(q^{-2}-q^{2}y_{i}^{4})}{(1-q^{-2})(1-y_{i}^4)}
\prod_{j=1}^{k}\frac{q^2(1-q^{-2}y_{i}^{2}y_{j}^{-2})(1-y_{i}^{2}y_{j}^{2})}{(1-q^{2}y_{i}^{2}y_{j}^{-2})(q^{-2}-q^{2}y_{i}^{2}y_{j}^{2})}
\end{equation}
\[
\times\prod_{l=1}^{n}\frac{(1-q^{2}y_{i}^{2}t_{l}^{2})(1-q^{2}y_{i}^{2}t_{l}^{-2})}{q^2(1-y_{i}^{2}t_{l}^{2})(1-y_{i}^{2}t_{l}^{-2})}
=
1.
\]
\end{Prop}
\begin{proof}
By the following substitution,
\begin{align}
s_{j}:=\frac{y_{j}^2}{\sqrt{q}},&&    a_{l}:=t_{l}^2, &&\mu:=\frac{\xi}{\sqrt{q}}, &&\tilde\mu:=\frac{\tilde\xi}{\sqrt{q}},&&q:=q^{-2},
\end{align}
one can see that the Bethe equations coincide with \eqref{bethe from VW} of the following form
\begin{align}
-\frac{s_i^2-q}{q s_i^2-1}\frac{ \left(\mu  q s_i-1\right) \left(\tilde \mu q s_i-1\right)}{ \left(s_i-\mu  q\right)
   \left(s_i-\tilde \mu q\right)}    \prod_{j=1}^k \frac{\left(q s_i-s_j\right) \left(q s_i s_j-1\right)}{\left(s_i-q s_j\right) \left(s_i s_j-q\right)} \prod_{l=1}^n \frac{\left(\sqrt{q} a_n-s_i\right) \left(a_n s_i-\sqrt{q}\right)}{\left(a_n-\sqrt{q} s_i\right) \left(\sqrt{q} a_n
   s_i-1\right)}=1.
\end{align}

\end{proof}

\subsection{XXZ \texorpdfstring{$U_q(\hat{\mathfrak{sl}}_2)$}{U_q(sl_2)} Open Chain by Yang-Nepomechie-Zhang \cite{Yang:2005ce}}
The Bethe equations of XXZ $U_q(\hat{\mathfrak{sl}}_2)$ Open Chain with non-diagonal boundary conditions are obtained in a different way compared to \cites{sklyaninboundary,Vlaar:2020jww}. In \cite{Yang:2005ce}, they defined Baxter's $Q$-operators by taking $j\to\infty$ limit of spin-$j$ transfer matrix. The TQ-relation is derived by the fusion hierarchy of transfer matrix of open spin-chain. The reduction to \cites{sklyaninboundary,Vlaar:2020jww} from \cite{Yang:2005ce}  is not obvious. 

In the following proposition, it is shown that the Bethe equations in \cite{Yang:2005ce} fit into the oper construction. 
\begin{Prop}
The generalized XXZ Bethe equations given by \eqref{bethe from VW} reduce to the open XXZ Bethe equations given in \cite[Equation 4.21]{Yang:2005ce}:
\begin{equation}\label{nepomechie}
    \frac{H_{2}^{(+)}(v_{i}|\epsilon_{1},\epsilon_{2},\epsilon_{3})}{H_{2}^{(+)}(-v_{i}-\eta|\epsilon_{1},\epsilon_{2},\epsilon_{3})}
    =
    -\frac{Q^{(+)}(v_{i}+\eta)}{Q^{(+)}(v_{i}-\eta)},
\end{equation}
where $\{v_{i}\}_{i=1,\ldots,M^{(+)}}$ are Bethe roots and the functions $Q^{(+)}$ and $H_{2}^{(+)}$ are given as follows:
\begin{equation}
    Q^{(+)}(z)=
\prod_{j=1}^{M^{(+)}}
\sinh(z-v_{j})\sinh(z+v_{j}+\eta)\,,
\end{equation}
and
\begin{align}
H_{2}^{(+)}(z|\epsilon_{1},\epsilon_{2},\epsilon_{3})
&=-4\epsilon_{2}\sinh^{2N}(z+\eta)\frac{\sinh(2z+2\eta)}{\sinh(2z+\eta)}\cr
&\cdot\sinh(z-\alpha_{-})\cosh(z-\epsilon_{1}\beta_{-})\sinh(z-\epsilon_{2}\alpha_{+})\cosh(z-\epsilon_{3}\beta_{+}),
\end{align}
where $\epsilon_{i}\in\{1,-1\}, 1\leq i\leq 3$ and $\alpha_{\pm}$ and $\beta{\pm}$ are boundary parameters satisfying certain constraints. Although there are additional parameters in \cite{Yang:2005ce}, they are all expressed in terms of the parameters due to a constraint among them. For more detail, see \cite{Yang:2005ce}.

\end{Prop}
\begin{proof}
Indeed, using the substitution,
\begin{align}
    q:=e^{-2\eta}, &&s_{j}:=e^{2v_{j}+\eta},
\end{align}
one sees that \eqref{nepomechie} coincides with the Bethe equations \eqref{bethe from VW} constructed geometrically via $q$-opers in Section \ref{GL2} (cf. Section \ref{simple pole at zero}). These equations take the following form:

\begin{align}
&\frac{\left(q-s_i^2\right) \left(e^{-2 \alpha _-} \sqrt{q}
   s_i-1\right) \left(\sqrt{q} e^{-2 \alpha _+ \epsilon _2}
   s_i-1\right) \left(\sqrt{q} e^{-2 \beta _- \epsilon _1}
   s_i+1\right) \left(\sqrt{q} e^{-2 \beta _+ \epsilon _3}
   s_i+1\right)}{\left(q s_i^2-1\right) \left(e^{-2 \alpha _-}
   \sqrt{q}-s_i\right) \left(\sqrt{q} e^{-2 \alpha _+ \epsilon
   _2}-s_i\right) \left(s_i+\sqrt{q} e^{-2 \beta _- \epsilon
   _1}\right) \left(s_i+\sqrt{q} e^{-2 \beta _+ \epsilon
   _3}\right)}\nonumber
   \\
   &\times\frac{(\sqrt{q}^{-1}s_{i}-1)^{2N}}{(\sqrt{q}s_{i}-1)^{2N}}q^{N}=-\prod_{j=1}^{M^{(+)}}\frac{(s_{i}-qs_j)\left(s_{i}s_j-q\right)}{(qs_{i}-s_j)\left(qs_{i}s_j-1\right)}.
\end{align}

\end{proof}

\subsection{XXZ \texorpdfstring{$U_q(\hat{\mathfrak{sl}}_N)$}{U_q(sl_N)} Open Chain by De Vega--Gonzales-Ruiz \cite{deVega:1994sb}}
In \cite{deVega:1994sb}, Bethe equations of XXZ $U_q(\hat{\mathfrak{sl}}_n)$ Open Chain with diagonal boundary conditions are obtained using nested algebraic Bethe Ansatz. In this section, we will discuss the connection with the Bethe equations for $U_q(\hat{\mathfrak{sl}}_n)$ open spin chain derived by De Vega and Gonzales-Ruiz \cite{deVega:1994sb} 
\begin{equation}\label{devega1}
    h^{(k)}(\mu_{i}^{(k)})\prod_{j\neq i}^{p_{k}}
    \frac{\sinh\left(\mu_{i}^{(k)}+\mu_{j}^{(k)}+(k-1)\gamma\right)\sinh\left(\mu_{i}^{(k)}-\mu_{j}^{(k)}-\gamma\right)}{\sinh\left(\mu_{i}^{(k)}+\mu_{j}^{(k)}+(k+1)\gamma\right)\sinh\left(\mu_{i}^{(k)}-\mu_{j}^{(k)}+\gamma\right)}
    =
    \end{equation}
    \[
\prod_{j=1}^{p_{k+1}}\frac{\sinh\left(\mu_{i}^{(k)}+\mu_{j}^{(k+1)}+k\gamma\right)\sinh\left(\mu_{i}^{(k)}-\mu_{j}^{(k+1)}-\gamma\right)}{\sinh\left(\mu_{i}^{(k)}+\mu_{j}^{(k+1)}+(k+1)\gamma\right)\sinh\left(\mu_{i}^{(k)}-\mu_{j}^{(k+1)}\right)}
\]
\[
    \prod_{j=1}^{p_{k-1}}\frac{\sinh\left(\mu_{i}^{(k)}+\mu_{j}^{(k-1)}+(k-1)\gamma\right)\sinh\left(\mu_{i}^{(k)}-\mu_{j}^{(k-1)}\right)}{\sinh\left(\mu_{i}^{(k)}+\mu_{j}^{(k-1)}+k\gamma\right)\sinh\left(\mu_{i}^{(k)}-\mu_{j}^{(k-1)}+\gamma\right)}
\]
where $1\leq k \leq N-1, 1\leq i\leq p_{k}$, $\mu_{i}^{(k)}$ are Bethe roots and the functions $h^{(k)}(\theta)$ are given as follows:
\[
l_{+}\neq l_{-},\quad h^{(l_{-})}(\mu_{i}^{(l_{-})})=
\frac{\sinh(\xi_{-}-\mu_{i}^{(l_{-})})}{\sinh(\xi_{-}+\mu_{i}^{(l_{-})}+l_{-}\gamma)}e^{2\mu_{i}^{(l_{-})}+l_{-}\gamma},
\]
\[
h^{(l_{+})}(\mu_{i}^{(l_{+})})=
\frac{\sinh(\xi_{+}+\mu_{i}^{(l_{+})}-(N-l_{+})\gamma)}{\sinh(\xi_{+}-\mu_{i}^{(l_{+})}-N\gamma)}e^{-2\mu_{i}^{(l_{+})}-l_{+}\gamma},
\]
\[
l_{+}=l_{-}=l, \quad 
h^{(l)}(\mu_{i}^{(l)})
=
\frac{\sinh(\xi_{-}-\mu_{i}^{(l)})}{\sinh(\xi_{-}+\mu_{i}^{(l)}+l\gamma)}
\frac{\sinh(\xi_{+}+\mu_{i}^{(l)}-(N-l)\gamma)}{\sinh(\xi_{+}-\mu_{i}^{(l)}-N\gamma)},
\]
\[
k\neq l_{+},l_{-}, \quad
h^{(k)}(\theta)=1.
\]

\subsubsection{Comparing with Our Results}
In order to provide a direct match between the above formulae and our Bethe equations we need to change the $q$-shift convention for $Q_i^\pm(z)$ and $\Lambda(z)$. In particular,
\begin{equation}\label{eq:NewqShiftQQ}
    Q_{k}^{\pm}(z)\in\mathbb{C}\left[z+\frac{1}{q^k z}\right],\quad k=1,\dots,N-1,\qquad \Lambda(z)\in \mathbb{C}\left[z+\frac{1}{q^N z}\right]
\end{equation}
In other words
\begin{equation}
Q_{k}^{+}(z)=
\gamma_{k}\prod_{j=1}^{p_{k}}\left(z-s_{k,j}\right)\left(\frac{1}{q^{k}z}-s_{k,j}\right), \qquad \gamma_{k}\in\mathbb{C}.
\end{equation}

With this adjustment in mind, we get a new version of \eqref{eq:bethe} with the new $q$-shift convention.
\begin{align}\label{eq:BetheEqnsGLNwithnewconvention}
    x_k(s_{k,i}) \cdot
    q^{p_{k}-p_{k-1}}
    \cdot&
    \prod_{j=1}^{p_{k-1}}\frac{qs_{k,i}-s_{k-1,j}}{s_{k,i}-s_{k-1,j}}\frac{q^{k}s_{k,i}s_{k-1,j}-1}{q^{k-1}s_{k,i}s_{k-1,j}-1}    
    \cr
    \cdot&\prod_{j=1}^{p_{k}}\frac{s_{k,i}-qs_{k,j}}{qs_{k,i}-s_{k,j}
    }\frac{q^{k-1}s_{k,i}s_{k,j}-1}{q^{k+1}s_{k,i}s_{k,j}-1}\cr
    \cdot&
\prod_{j=1}^{p_{k+1}}\frac{s_{k,i}-s_{k+1,j}}{s_{k,i}-qs_{k+1,j}}
\frac{q^{k+1}s_{k,i}s_{k+1,j}-1}{q^{k}s_{k,i}s_{k+1,j}-1}
=-1\,,    
\end{align}
for $k=1,\dots,N-1$ where $p_0=0$, and $s_{N,i}=a_i$.
The above equations match \eqref{devega1} after applying the following exponential change of variables\begin{equation}
    q=e^{2\gamma}, 
\qquad 
s_{k,i}=e^{2\mu_{i}^{(k)}},
\qquad
b_{\pm}=e^{2\xi_{\pm}},
\end{equation}
provided that
\begin{align}\label{eq:gentwist devega}
 x_k(s_{k,i}) =  q^{-1}h^{(k)}(\mu_{i}^{(k)})\frac{(q^{k+1}s_{k,i}^{2}-1)}{(q^{k-1}s_{k,i}^{2}-1)}\,.
\end{align}
The twist functions $h^{(k)}(z)$ take the following form in our variables
\begin{align}\label{eq:hdevega}
h^{(l_{-})}(z)
&= q^{l_{-}}\,z\,\frac{b_{-}-z}{q^{l_{-}}b_{-}z-1},
\qquad h^{(l_{+})}(z)
= \frac{1}{z}\,\frac{b_{+}z-q^{\,N-l_{+}}}{b_{+}-q^{N}z},\qquad l_{+}\neq l_{-}\cr
\qquad 
h^{(l)}(z)
&= q^{l}\,\frac{b_{-}-z}{q^{l}b_{-}z-1}\,
\frac{b_{+}z-q^{\,N-l}}{b_{+}-q^{N}z}\,,\qquad l_{+}=l_{-}=l,
\end{align}
as well as $h^{(l)}(z)=1$ for $l\neq l_+,l_-$.

\vskip.1in
Notice that once the convention \eqref{eq:NewqShiftQQ} is fixed the structure of the equations \eqref{devega1} is completely reproduced up to the twist factors which in our language follows from the solutions of the reflection equations which we have discussed in Section \ref{Sec:SolutionsReflEqns}.

Below we discuss low rank examples in more detail.

\subsubsection{Twists from $(GL(2),q)$-Opers}\label{sl2exampledevega}
    Let us consider the case where $N=2$. In this case, $l_{-}=l_{+}=1$ and $x_{1}(z)$ from \eqref{eq:gentwist devega} takes the following form
\[
x_{1}(z)
=
\frac{(b_{-}-z)}{(qb_{-}z-1)}\frac{(b_{+}z-q)}{(b_{+}-q^{2}z)}
\frac{(q^{2}z^{2}-1)}{(z^{2}-1)}.
\]
Now we show that $x_{1}(z)$ coming from De Vega and Gonzales-Ruiz's equations for $N=2$ is indeed a part of the following solution $(x'_{1}(z),x'_{2}(z))$ of the reflection equation \eqref{eq:x1GL2} (in the new convention \eqref{eq:NewqShiftQQ})
\[
x_{1}(z)=x'_{1}(z)=-\frac{x'_{2}(z)}{x'_{2}(\frac{1}{qz})}
\]
where
\[
x'_{2}(z)=\left(b_{-}-z\right)\left(b_{+}z-q\right)\left(1-\frac{1}{q^{2}z^{2}}\right).
\]

\subsubsection{Twists from $(GL(3),q)$-Opers}\label{sl3exampledevega}
We present the Bethe equation \eqref{eq:gentwist devega} explicitly with \eqref{eq:hdevega} for all possible pairs $(l_+,l_-)$:
\begin{enumerate}
\item $l_\pm=l=1$ 
\begin{align}
x_1(z) = \,\frac{b_{-}-z}{qb_{-}z-1}\,
\frac{b_{+}z-q^{2}}{b_{+}-q^{3}z}\frac{((qz)^{2}-1)}{(z^{2}-1)},
&&   x_2(z)=\frac{q^3 z^{2}-1}{q^2z^{2}-q}.
\end{align}
Taking $a=b_- q^2$, $c=\frac{q^4}{b_+}$, the solution \eqref{eq:typeD} reproduces the desirable structure of poles and zeros as above when $b_+=q^8 b_-$, albeit with slightly different $q$-shifts.
\begin{align}
&x_1(z)=\frac{\left(q^4 z^2-1\right) }{\left(z^2-q^4\right)}\cdot\frac{\left(q^7 z-b_+\right) \left(q^3-b_+ z\right)}{\left(q^3 z-b_+\right) \left(q^7-b_+ z\right)},
\\
&x_2(z)=\frac{q \left(z^2-q^2\right)}{q^4 z^2-1}.
\end{align}

\item $l_\pm=l=2$
\begin{align}
x_1(z)=\frac{(q^{2}z^{2}-1)}{q(z^{2}-1)},
&& 
x_2(z) =q\,\frac{b_{-}-z}{q^{2}b_{-}z-1}\,
\frac{b_{+}z-q}{b_{+}-q^{3}z}\frac{(q^3 z^{2}-1)}{(qz^{2}-1)}.
\end{align}

This solution can be reconstructed in a similar way from our results in Section 4

\item $l_-=1$, $l_+=2$
\begin{align}
x_1(z) =  \,z\,\frac{b_{-}-z}{qb_{-}z-1}\frac{(q^{2}z^{2}-1)}{(z^{2}-1)},
&&
x_2(z) =  \frac{1}{qz}\,\frac{b_{+}z-q}{b_{+}-q^{3}z}\frac{(q^{3}z^{2}-1)}{(qz^{2}-1)}.
\end{align}
\item $l_-=2$, $l_+=1$
\begin{align}
x_1(z) = \frac{1}{qz}\,\frac{b_{+}z-q^{2}}{b_{+}-q^{3}z}\frac{(q^{2}z^{2}-1)}{(z^{2}-1)},
&&
x_2(z) = q\,z\,\frac{b_{-}-z}{q^{2}b_{-}z-1}\frac{(q^{3}z^{2}-1)}{(qz^{2}-1)}.
\end{align}
\end{enumerate}

In the current setting, our solutions cannot reproduce the results of de Vega-Gonzalez-Ruiz for $l_-\neq l_+$. We expect that after certain modifications of our construction of reflection-invariant opers, i.e. the condition $s(z^{-1})=s(z)$, these cases can be recovered. This is however, beyond the scope of the current manuscript.

\subsection{XXX \texorpdfstring{$Y(\mathfrak{sl}_2)$}{Y(sl_2)} Open Chain by Frassek-Szecsenyi \cite{Frassek:2015aa}}

The Bethe Ansatz equations for open XXX spin chains discovered in \cite{Frassek:2015aa} are given as:
\begin{equation}\label{frassek}
    -\frac{2z_{i}}{(2z_{i}+2)}\frac{(z_{i}-p+1)(z_{i}-q+1)}{(z_{i}+p)(z_{i}+q)}\times \prod_{k=1}^{m}\frac{(z_{i}-z_{k}+1)(z_{i}+z_{k}+2)}{(z_{i}-z_{k}-1)(z_{i}+z_{k})}
    \times
    \frac{(z_{i})^{2L}}{(z_{i}+1)^{2L}}=1, \quad 1\leq i\leq m,
\end{equation}
where $z_{i}$ are Bethe roots and $p$ and $q$ are boundary parameters.
In the equations \ref{bethe using additive opers}, if we let $\epsilon=1$ and all $a_{j}$'s to be $1/2$ and change variables as $z_{k}=s_{k}-\frac{1}{2}$, then we recover \eqref{frassek} using the geometric construction via $\epsilon$-opers in Section \ref{epsilon opers}.

\bibliography{cpn1}
\end{document}